\documentclass[12pt,a4paper]{article}
\usepackage[utf8]{inputenc}
\usepackage[UKenglish]{isodate}
\cleanlookdateon
\usepackage{amsmath}
\usepackage{amsthm}
\usepackage{amssymb}
\usepackage{enumitem}
\usepackage[T1]{fontenc}
\usepackage{mathtools}
\usepackage{xcolor}
\usepackage{caption}
\usepackage{subcaption}
\usepackage{hyperref}
\usepackage{booktabs}
\usepackage{tikz}
\usetikzlibrary{math}
\tikzstyle{vertex}=[circle, draw, fill=black, inner sep=0pt, minimum width=4pt]
\usepackage{thmtools}
\usepackage{thm-restate}

\hypersetup{
  colorlinks   = true,
  urlcolor     = blue!50!black,
  linkcolor    = blue!50!black,
  citecolor   = blue!50!black
}

\newcommand\blfootnote[1]{
  \begingroup
  \renewcommand\thefootnote{}\footnote{#1}
  \addtocounter{footnote}{-1}
  \endgroup
}

\title{Improved bounds for 1-independent percolation on $\Z^n$}
\author{Paul Balister\thanks{Mathematical Institute, University of Oxford, Oxford, OX2\thinspace6GG, UK.}
\thanks{Partially supported by EPSRC grant EP/W015404/1.}
\and Tom Johnston\thanks{School of Mathematics, University of Bristol, Bristol, BS8\thinspace1UG, UK.} \thanks{Heilbronn Institute for Mathematical Research, Bristol, UK.}
\and Michael Savery\footnotemark[1] \footnotemark[4]
\and Alex Scott\footnotemark[1] \thanks{Partially supported by EPSRC grant EP/V007327/1.}}

\date{\today}

\newtheorem{theorem}{Theorem}[section]
\newtheorem{lemma}[theorem]{Lemma}
\newtheorem{proposition}[theorem]{Proposition}
\newtheorem{corollary}[theorem]{Corollary}
\newtheorem{result}[theorem]{Result}
\newtheorem{claim}[theorem]{Claim}
\newtheorem{conjecture}[theorem]{Conjecture}

\newtheorem{problem}[theorem]{Problem}
\theoremstyle{definition}
\newtheorem{definition}[theorem]{Definition}

\newcommand\ceil[1]{\left \lceil #1 \right \rceil}
\newcommand\floor[1]{\left \lfloor #1 \right \rfloor}
\newcommand\eps{\varepsilon}
\newcommand\N{\mathbb{N}}
\newcommand\Z{\mathbb{Z}}
\newcommand\E{\mathbb{E}}
\newcommand\Prb{\mathbb{P}}
\newcommand\pmax{p_{\mathrm{max}}}
\newcommand\psite{p_{\mathrm{site}}}
\newcommand\pgiant{p_{\mathrm{giant}}}
\newcommand\cD{\mathcal{D}}
\newcommand\cC{\mathcal{C}}
\newcommand\cL{\mathcal{L}}

\newcommand\concat{\mathbin{\|}}

\advance\textwidth2\evensidemargin
\begin{document}

\maketitle

\begin{abstract}
    
A 1-independent bond percolation model on a graph $G$ is a probability distribution on the spanning subgraphs of $G$ in which, for all vertex-disjoint sets of edges $S_1$ and $S_2$, the states of the edges in $S_1$ are independent of the states of the edges in $S_2$. Such a model is said to percolate if the random subgraph has an infinite component with positive probability. In 2012 the first author and Bollob\'as defined $\pmax(G)$ to be the supremum of those $p$ for which there exists a 1-independent bond percolation model on $G$ in which each edge is present in the random subgraph with probability at least $p$ but which does not percolate.
    
A fundamental and challenging problem in this area is to determine the value of $\pmax(G)$ when $G$ is the lattice graph $\Z^2$. Since $\pmax(\Z^n)\leq \pmax(\Z^{n-1})$, it is also of interest to establish the value of $\lim_{n\to\infty}\pmax(\Z^n)$. In this paper we significantly improve the best known upper bound on this limit and obtain better upper and lower bounds on $\pmax(\Z^2)$. In proving these results, we also give an upper bound on the critical probability for a 1-independent model on the hypercube graph to contain a giant component asymptotically almost surely.\blfootnote{Email: \{\href{mailto:balister@maths.ox.ac.uk}{\texttt{balister}}, \href{mailto:savery@maths.ox.ac.uk}{\texttt{savery}}, \href{mailto:scott@maths.ox.ac.uk}{\texttt{scott}}\}\texttt{@maths.ox.ac.uk},
\href{mailto:tom.johnston@bristol.ac.uk}{\texttt{tom.johnston@bristol.ac.uk}.}
}
\end{abstract}

\section{Introduction}

A \emph{percolation model} on a (possibly infinite) graph $G$ is a probability distribution on the subgraphs of~$G$. We say that a vertex or edge of $G$ is \emph{open} if it is present in the random subgraph of $G$ associated with the percolation model, and $\emph{closed}$ otherwise. A \emph{bond percolation model} on $G$ is a probability distribution on the spanning subgraphs of~$G$, meaning all vertices are open and the random subgraph depends only on the edges. An \emph{independent} bond percolation model on $G$ is a bond percolation model in which each edge is open independently. The focus of this paper is on the weaker condition of \emph{1-independence}: a bond percolation model on $G$ is said to be \emph{1-independent} if, for any two disjoint sets of edges $S_1$ and $S_2$ of $G$ such that no edge in $S_1$ shares a vertex with an edge in $S_2$, the states (i.e.\ open or closed) of the edges in $S_1$ are independent of the states of the edges in~$S_2$. If $G$ is an infinite graph, then we say that a percolation model on $G$ \emph{percolates} if the associated random subgraph of $G$ has an infinite component with positive probability. Note that by a slightly modified version of Kolmogorov's zero-one law, if a percolation model on a locally-finite infinite graph $G$ percolates, then in fact the random subgraph has an infinite component with probability one.

The study of 1-independent bond percolation models is motivated by their use as a tool to obtain bounds on critical probabilities of independent models via renormalisation (see, for example,~\cite{balister2005continuum} and Sections 3.5 and 6.2 of~\cite{bollobas2006percolation}). They have also been combined with renormalisation techniques to analyse bond percolation models with dependencies over a greater range (see~\cite{balister2005continuum} again for an example).

For an infinite graph~$G$, we wish to obtain results on whether or not a 1-independent model on $G$ in which each edge is open with some probability $p\in(0,1)$ percolates. For any given $p$ there will be (uncountably) many different 1-independent models, but we might expect that if $p$ is sufficiently large, then any 1-independent model on $G$ will percolate. In this spirit, in~\cite{balister2012critical} the first author and Bollob\'as defined $\cD_{\geq p}(G)$ to be the class of 1-independent bond percolation models on $G$ in which each edge is open with probability at least $p$. They then defined
\[
 \pmax(G)=\sup\big\{p\,:\,\text{some model in }
 \cD_{\geq p}(G)\text{ does not percolate}\big\}.
\]

Note that in this definition it is equivalent to ask for a model in which each edge is open with probability exactly $p$. Indeed, if $\mu\in\cD_{\geq p}(G)$ does not percolate, then the model in which we sample from $\mu$ and then independently delete each edge $e$ with probability $1-p/\Prb_\mu(e\text{ open})$ is a 1-independent bond percolation model which does not percolate, and in which each edge is open with probability exactly~$p$.

The value of $\pmax(G)$ for different infinite graphs~$G$ has been studied, along with various other properties of 1-independent bond percolation models, in~\cite{balister2012critical,balister2005continuum,day2020long,falgas-ravry20231-independent,liggett1997domination}. In particular, in~\cite{balister2012critical} the authors determine $\pmax(T)$ for all locally finite infinite trees $T$ in terms of a parameter known as the \emph{branching number} of~$T$. They go on to establish that $\pmax(G)\geq\frac12$ for all locally finite connected infinite graphs~$G$, and they construct such a $G$ which achieves this lower bound.

The case where $G$ is the lattice graph $\Z^2$ has been of particular interest, with the first author and Bollob\'as asking for the value of $\pmax(\Z^2)$.
\begin{restatable}[\cite{balister2012critical}]{question}{questionZtwo}\label{question:Z2}
 What is the value of $\pmax(\Z^2)$?
\end{restatable}

Considering $\Z^n$ more generally, they noted that $\pmax(\Z^n)$ is decreasing in $n$ (since if a model does not percolate on $\Z^n$ then it does not percolate on any $\Z^{n-1}$ subspace), and asked for the limit of this value as $n\to\infty$.
\begin{restatable}[\cite{balister2012critical}]{question}{questionZn}\label{question:Zn}
 What is the value of $\lim_{n\to\infty}\pmax(\Z^n)$?
\end{restatable}

The best known lower bound on both $\pmax(\Z^2)$ and $\lim_{n\to\infty}\pmax(\Z^n)$ comes from a construction of Day, Falgas-Ravry, and Hancock in~\cite{day2020long} which gives
\[
 \pmax(\Z^n)\geq 4-2\sqrt{3}= 0.535898\dots
\]
for all $n\geq 2$. It was conjectured by Falgas-Ravry and Pfenninger in~\cite{falgas-ravry20231-independent} that there exists some $n\geq 3$ for which this is an equality. In~\cite{day2020long} the authors also give a separate construction which yields
\[
 \pmax(\Z^2)\geq \psite^2+\tfrac{1}{2}(1-\psite),
\]
where $\psite=\psite(\Z^2)$ denotes the critical probability for independent site percolation on $\Z^2$, that is, the supremum of the $p$ such that percolation does not occur in the percolation model on $\Z^2$ in which vertices are open independently with probability $p$ and edges are open if both their endpoints are open. The best known rigorous lower bound on $\psite$ is $0.556$ due to van den Berg and Ermakov~\cite{berg1996new}. Substituting this into the above gives $\pmax(\Z^2)\geq0.531136$ which is slightly worse than the bound given by the first construction. However, using the non-rigorous estimate $\psite\approx 0.592746$ (see, for example,~\cite{ziff1992spanning}) gives an improved bound of $\pmax(\Z^2)\geq 0.554974$.

The previous best known upper bound on $\pmax(\Z^2)$ was $0.8639$ due to the first author, Bollob\'as, and Walters in~\cite{balister2005continuum}, and no better upper bounds on $\lim_{n\to\infty}\pmax(\Z^n)$ were known. In this paper, we improve both upper bounds as follows.

\begin{theorem}\label{thm:z2-upper-bound}
 $\pmax(\Z^2)\le 0.8457$.
\end{theorem}

\begin{theorem}\label{thm:zn-limit}
 $\lim_{n\to\infty}\pmax(\Z^n)\le 0.5847$.
\end{theorem}

We also improve the lower bound for $\pmax(\Z^2)$.

\begin{theorem}\label{thm:z2-lower-bound}
 $\pmax(\Z^2)\ge \tfrac{1}{32}(35-3\sqrt{33})= 0.555197\dots$
\end{theorem}

Note that this new bound is an improvement even on the non-rigorous bound stated above. However, it is likely that the true value of $\pmax(\Z^2)$ is still some distance from this, as suggested by the following high confidence result which gives an even better bound (see Section~\ref{sec:z2-lower-bounds}).

\begin{result}\label{res:high-conf}
With high confidence
\textup{(}$p$-value $<10^{-11}$\textup{)} we have
$\pmax(\Z^2)\ge0.5921$.
\end{result}

The proof of Theorem~\ref{thm:z2-upper-bound} uses a $2\times 2$ renormalisation argument similar to that employed in~\cite{balister2005continuum}.
We obtain a better bound than was derived in that paper by using a different condition for two renormalised sites to be joined by an edge. To calculate a lower bound on the probability that each edge is present, we solve two linear programming problems that are relaxations of the 1-independence constraints on a $4\times 2$ grid.

The proof of Theorem~\ref{thm:zn-limit} relies on the following result which translates a condition on 1-independent bond percolation models on the hypercube to one on the lattice. Here we denote by $Q_n$ the $n$-dimensional hypercube graph which has vertex set $\{0,1\}^n$ and in which vertices are joined by an edge if they differ in exactly one coordinate.

\begin{theorem}\label{thm:lattice}
 Let $p\in(0,1]$ and\/ $k\in\N$ be constants and let $P>(1-p)^{2^{k-1}}\varphi$, where $\varphi=\frac{1}{2}(\sqrt{5}+1)$ is the golden ratio.
 Suppose that for any model in $\cD_{\ge p}(Q_k)$, the probability that the graph is connected is at least~$P$. Then for large enough~$n$, every model in $\cD_{\ge p}(\Z^n)$ percolates. In particular,
 \[\lim_{n\rightarrow\infty} \pmax(\Z^n)\leq p.\]
\end{theorem}

In order to extract a concrete bound from this result we need the following lemma, which states that $p=0.5847$ and $k=6$ satisfy the conditions of the theorem.

\begin{lemma}\label{lem:q6-p}
 Let $\varphi$ be the golden ratio. Then there exists a constant $P>(1-0.5847)^{32}\varphi$ such that in any model
 in $\cD_{\ge 0.5847}(Q_6)$, the probability that the graph is connected is at least~$P$.
\end{lemma}

Together, these results give Theorem~\ref{thm:zn-limit}. The approach we use to prove Theorem~\ref{thm:lattice} also yields the following result about the critical probability for the asymptotically almost sure existence of a giant component in the hypercube graph $Q_n$ under 1-independent bond percolation models, which may be of independent interest.

\begin{theorem}\label{thm:cube}
 Let $p\in(0,1]$ and\/ $k\in\N$ be constants and let $P>(1-p)^{2^{k-1}}\varphi$, where $\varphi$ is the golden ratio. Suppose that for any model in $\cD_{\ge p}(Q_k)$, the probability that the graph is connected is at least~$P$. Then there exists $C>0$ such that for $n\in\N$, and any model in $\cD_{\ge p}(Q_n)$, there is a component containing at least $C\cdot 2^n$ vertices with probability at least $1-e^{-\Omega(n^2)}$.
\end{theorem}

Motivated by the statement of this theorem, let
\[
 \pgiant=\inf\Big\{p:\exists C>0\colon
 \!\lim_{n\to\infty}\inf_{\mu\in\cD_{\ge p}(Q_n)}\!\Prb_\mu(Q_n\text{ has a component of size}\ge C\cdot 2^n)=1\Big\}
\]
be the threshold $p$ for the asymptotically almost sure existence of a giant component in~$Q_n$ under any model in $\cD_{\geq p}(Q_n)$. Then Theorem~\ref{thm:cube}
together with Lemma~\ref{lem:q6-p} imply that
$\pgiant \le 0.5847$. The best known lower bound on $\pgiant$ comes from the simple construction used in the proof of Theorem 1.4 in~\cite{balister2012critical}, which shows that $\pgiant\geq\frac{1}{2}$.

It is interesting to compare the known bounds on these 1-independent threshold probabilities with their analogues in the independent setting. Firstly, the celebrated Harris--Kesten theorem~\cite{harris1960lower,kesten1980critical} determines the threshold for percolation in the independent bond model on $\Z^2$ to be exactly~$1/2$, whereas Result~\ref{res:high-conf} suggests that $\pmax(\Z^2)\geq 0.5921$. Turning to $\pgiant$, Erd\H{o}s and Spencer showed in~$\cite{erdos1979evolution}$ that for each fixed $\eps>0$, the random subgraph of $Q_n$ in which each edge is open independently with probability $p=(1-\eps)/n$ has maximum component size $o(2^n)$ asymptotically almost surely, while Ajtai, Koml\'os, and Szemer\'edi showed in~\cite{ajtai1982largest} that under the same model with $p=(1+\eps)/n$ where $\eps>0$ is fixed, there is a component of size $\Omega(2^n)$. The component structure has since been studied when $\eps\to0$ at different rates, for example in~\cite{bollobas1992evolution, borgs2005random1, borgs2005random2, borgs2006random, hofstad2005asymptotic,hofstad2006expansion} (see also~\cite{hofstad2014unlacing} or Section 13.4 of~\cite{heydenreich2017progress} for a survey). The threshold probability for percolation of the independent model on $\Z^n$ is $(1+o(1))/2n$; see~\cite{heydenreich2017progress} (particularly Section 15.5) for the history of this fact and a survey of more precise results.

After outlining some notation and preliminary results, in Section~\ref{sec:cube} we prove Theorem~\ref{thm:cube} via a series of propositions and lemmas. That section contains many of the key results and ideas used in the proof of Theorem~\ref{thm:lattice}, which is dealt with in Section~\ref{sec:lattice}. The proof of Lemma~\ref{lem:q6-p} is handled in Section~\ref{sec:q3-p}, and Section~\ref{sec:z2-lower-bounds} contains the proof of Theorem~\ref{thm:z2-lower-bound} and the explanation of Result~\ref{res:high-conf}. We prove Theorem~\ref{thm:z2-upper-bound} in Section~\ref{sec:z2-upper-bound}, and finally in Section~\ref{sec:open-probs} we discuss some possible avenues for future enquiry.

\subsection{Notation and preliminaries}
We use standard graph theoretic notation throughout. Although we often work on subgraphs of $Q_n$, the distance between two vertices will always be the Hamming distance, or equivalently, the distance between the vertices in $Q_n$. For a vertex $v$ of a subgraph $G$ of $Q_n$, define the second neighbourhood of $v$ in $G$ to be the set of vertices whose shortest path to $v$ in $G$ has length exactly 2, and write this as $N^2_G(v)$. 

For two tuples $u\in\Z^i$ and $v\in\{0,1\}^j$ where $i,j\in\N$, we will write $u\concat v$ to mean the concatenation of $u$ and $v$; if $j=1$, so that $v=(0)$ or $v=(1)$, then we will write $u\concat 0$ or $u\concat 1$ respectively. For $n\in \N$, we will write $[n]$ for the set $\{1,2,\dots,n\}$.

We now state the Chernoff bound we will use throughout this paper (see~\cite{mitzenmacher2005probability} for a discussion of this result).
\begin{lemma}[Chernoff bound]
    \label{lem:chernoff}
    Let $n\in\N$, let $p\in[0,1]$, and let $X\sim\textnormal{Bin}(n,p)$. Then for all $\eps\in[0,1]$,
    \[\Prb(X\leq (1-\eps)np)\leq e^{-\eps^2np/2}.\]
\end{lemma}

We also use the following simple corollary of Markov's inequality.
\begin{lemma}\label{lem:markov}
    Let $X$ be a random variable taking values in the interval $[0,N]$ for some $N\in(0,\infty)$, and suppose that $\mathbb{E}(X)\geq cN$ for some $c\in(0,1]$. Then for all $\eps\in(0,1)$, \[\Prb\big(X\geq (1-\eps)cN\big)\geq c\eps.\]
\end{lemma}
\begin{proof}
    Applying Markov's inequality we have
    \begin{align*}
        \Prb\big(X\geq (1-\eps)cN\big)
         & \geq 1-\Prb\big(N-X\geq N-(1-\eps)cN\big) \\
         & \geq 1-\frac{N-\mathbb{E}(X)}{N-(1-\eps)cN}     \\
         & = \frac{\mathbb{E}(X)/N-(1-\eps)c}{1-(1-\eps)c} \\
         & \geq \frac{c\eps}{1-(1-\eps)c}                  \\
         & \geq c\eps.\qedhere
    \end{align*}
\end{proof}

\section{Giant component in \texorpdfstring{$Q_n$}{Q\textunderscore{}n}}\label{sec:cube}
We begin by studying the giant component in the hypercube $Q_n$, and prove Theorem~\ref{thm:cube}. A key step in the proof is a renormalisation argument which reduces the problem to one concerning a more general class of percolation models where both the vertices and edges may be open or closed, but which still maintains that the states of subgraphs induced on disjoint subsets of vertices are independent. In return for allowing vertices to be closed, edges in the renormalised models have a higher probability of being open when both endpoints are open. Formally, we will consider percolation models from the following class.

\begin{definition}\label{def:renormalised}
    For a (possibly infinite) graph $G$, and for $p_v\in(0,1]$ and $p_e\in[0,1]$, define $\cD(G,p_v,p_e)$ to be the family of percolation models on $G$ in which \begin{enumerate}[label=(\roman*)]
        \item vertices are open with probability at least $p_v$;
        \item conditioned on their endpoints being open, edges are open with probability at least $p_e$; and
        \item for all pairs of vertex-disjoint subgraphs $G_1$ and $G_2$ of $G$, the states of the vertices and edges in
        $G_1$ are independent of the states of the vertices and edges in~$G_2$.
    \end{enumerate}
\end{definition}

The main result on such percolation models is the following proposition, proved in Section~\ref{sec:cube-prop}, from which Theorem~\ref{thm:cube} will follow almost immediately after a suitable renormalisation.

\begin{proposition}\label{prop:cube}
    Let $p_v\in (0,1]$ and\/ $p_e\in (\frac{18}{19},1]$ be constants. Then there exists a positive constant $C$ such that for all\/ $n\in\N$, under every model in $\cD(Q_n,p_v,p_e)$, there is a component containing at least $C\cdot 2^n$ vertices with probability at least $1-e^{-\Omega(n^2)}$.
\end{proposition}

In the next section we will prove Theorem~\ref{thm:cube} assuming Proposition~\ref{prop:cube}, then in Section~\ref{sec:cube-prop} we will return to prove the proposition via a series of lemmas. In order to apply Theorem~\ref{thm:cube} we will make use of Lemma~\ref{lem:q6-p}, which shows that the conditions of Theorem~\ref{thm:cube} hold when $p=0.5847$, and which is proved in Section~\ref{sec:q3-p}.

\subsection{Proof of Theorem~\ref{thm:cube}}\label{sec:cube-thm}

We now give the proof of Theorem~\ref{thm:cube} assuming Proposition~\ref{prop:cube} holds.

\begin{proof}[Proof of Theorem~\ref{thm:cube}]
Let $n\geq k$ be an integer. For a given 1-independent bond percolation model $\mu_n$ on $Q_n$ in which each edge is open with probability $p$, we will inductively define coupled percolation models $\mu_{n-k},\dots,\mu_1$ on $Q_{n-k},\dots,Q_{1}$ respectively, starting with $\mu_{n-k}$. We say a vertex $v \in Q_{n-k}$ is open under $\mu_{n-k}$ if the hypercube $H_v = \{v\concat a:a\in\{0,1\}^k\} \subseteq Q_n$ is connected under $\mu_n$. An edge of $Q_{n-k}$ between two open vertices $u$ and $v$ is chosen to be open under $\mu_{n-k}$ if at least one of the edges between the two hypercubes $H_u$ and $H_v$ is open under $\mu_n$ (so that the $(k+1)$-dimensional cube $H_u\cup H_v$ is connected). For $i\in \{k,\dots,n-2\}$, define $\mu_{n-i-1}$ to be the model on $Q_{n-i-1}$ in which a vertex $v$ is open if the edge between the (open) vertices $v\concat0$ and $v\concat1$ (in $Q_{n-i}$) is open under $\mu_{n-i}$, and an edge $uv$ between two open vertices is open if the edge between $u\concat0$ and $v\concat0$ or the edge between $u\concat1$ and $v\concat1$ is open.

For each $i\geq k$, define $q_i=(1-p)^{2^i}$, so that $q_{i+1}=q_i^2$ for these $i$. Let $s_k=P$, and for $i\geq k$, let $s_{i+1}=s_i^2-q_i$. Note that $s_i^2>q_i\varphi^2>0$ for all $i\geq k$, where $\varphi$ is the golden ratio. Indeed, $s_k^2=P^2>q_k\varphi^2$ by assumption, and if $s_{i-1}^2>q_{i-1}\varphi^2$ for $i\geq k+1$, then $s_i^2=(s_{i-1}^2-q_{i-1})^2>q_{i-1}^2(\varphi^2-1)^2=q_i\varphi^2$. Thus, for all $i\geq k$, we can define $r_i=q_i/s_i^2\in[0,1]$.

\begin{claim}\label{claim:renorm}
 For all\/ $k\leq i\leq n-1$, we have $\mu_{n-i}\in\cD(Q_{n-i},s_i,1-r_i)$.
\end{claim}
\begin{proof}
We prove the claim by induction on~$i$, considering first the case $i=k$. Since $\mu_n$ is a 1-independent bond percolation model, in $\mu_{n-k}$ vertices are open independently, and they are each open with probability at least $s_k=P$ by assumption. Let $uv$ be an edge of $Q_{n-k}$ and consider the hypercubes in $Q_n$ corresponding to $u$ and~$v$. Write $A$ for the event that all the edges in $Q_n$ between these hypercubes are closed. Since all these edges are open independently, the probability of $A$ occurring is $q_k$. Thus we have
\begin{align*}
 \Prb(uv \text{ open } \mid u,v \text{ open})
 &=\Prb(\{u, v \text{ open}\}\cap\{uv \text{ open}\} )/\Prb(u,v \text{ open}) \\
 &\geq (\Prb(u,v\text{ open})- \Prb(A))/\Prb(u,v \text{ open}) \\
 &\geq 1-q_k/P^2\\
 &=1-r_k.
\end{align*}

Finally, if $G_1$ and $G_2$ are vertex-disjoint subgraphs of $Q_{n-k}$, then the states of the vertices and edges in each are determined by the states of the edges in two vertex-disjoint subgraphs of~$Q_n$. Since $\mu_n$ is a 1-independent model, the states of the vertices and edges in $G_1$ and $G_2$ are independent of each other and hence $\mu_{n-k}\in\cD(Q_{n-k},s_k,1-r_k)$.

Now let $k+1\leq i\leq n-1$ and assume that $\mu_{n-i+1}\in\cD(Q_{n-i+1},s_{i-1},1-r_{i-1})$. This means each vertex in $Q_{n-i+1}$ is open independently with probability at least~$s_{i-1}$, and given that two endpoints of an edge of $Q_{n-i+1}$ are open, the edge itself is open with probability at least $1-r_{i-1}$. It follows that the probability that a vertex of $Q_{n-i}$ is open is at least $(1-r_{i-1})s_{i-1}^2=s_i$. If $G_1$ and $G_2$ are vertex-disjoint subgraphs of $Q_{n-i}$, then the states of the vertices and edges in each are determined by the states of the vertices and edges in two vertex-disjoint subgraphs of $Q_{n-i+1}$. Since $\mu_{n-i+1}\in\cD(Q_{n-i+1},s_{i-1},1-r_{i-1})$, we deduce that the states of the edges and vertices in $G_1$ and $G_2$ are independent of each other.

It just remains to check that an edge $uv$ of $Q_{n-i}$ is present with probability at least $1 - r_i$ given that $u$ and $v$ are open. Let $u_b=u\concat b$ and $v_b=v\concat b$ for each $b\in\{0,1\}$. The edge $uv$ is open in $Q_{n-i}$ exactly when the edges $u_0u_1$, $v_0v_1$, and at least one of $u_0v_0$ and $u_1v_1$ are open in $Q_{n-i+1}$. Let $B$ be the event that $u_0$, $u_1$, $v_0$ and $v_1$ are all open in $Q_{n-i+1}$. Then
\begin{align*}
 \Prb(uv \text{ open} \mid &\ u,v \text{ open})
  =\Prb(\{uv \text{ open}\}\cap\{ u, v \text{ open}\})/\Prb(u,v \text{ open})\\
 &= \Prb(\{u_0u_1, v_0v_1 \text{ open}\}\cap \{u_0v_0\text{ or } u_1v_1 \text{ open}\})/\Prb(u_0u_1, v_0v_1 \text{ open})\\
 & = \Prb(\{u_0u_1, v_0v_1 \text{ open}\}\cap \{u_0v_0\text{ or } u_1v_1 \text{ open}\}\mid B)/\Prb(u_0u_1, v_0v_1 \text{ open} \mid B)\\
 & \geq \frac{\Prb(u_0u_1, v_0v_1 \text{ open}\mid B )-\Prb(u_0v_0, u_1v_1 \text{ closed}\mid B)}{\Prb(u_0u_1, v_0v_1 \text{ open} \mid B )}\\
 & \geq 1 - \frac{r_{i-1}^2}{(1-r_{i-1})^2}\\
 & = 1 - \frac{q_{i-1}^2/s_{i-1}^4}{(s_{i-1}^2-q_{i-1})^2/s_{i-1}^4}\\
 & =1-r_i.
\end{align*}
Hence, $\mu_{n-i}\in\cD(Q_{n-i},s_{i},1-r_{i})$ as required.
\end{proof}

\begin{claim}
 We have $r_i\to 0$ as $i\to \infty$.
\end{claim}
\begin{proof}
For $i\ge k$, we have
\[
    \frac{1}{r_{i+1}}=\frac{s_{i+1}^2}{q_{i+1}}
    =\left(\frac{s_i^2-q_i}{q_i}\right)^2
    =\left(\frac{1}{r_i}-1\right)^2.
\]
If $1/r_i=\varphi^2+\eps$ for some
$\eps>0$, then
\[
    \frac{1}{r_{i+1}}=(\varphi^2-1+\eps)^2
    =(\varphi+\eps)^2\ge \varphi^2+2\varphi\eps.
\]
In particular, this can be iterated to obtain
\[
 \frac{1}{r_{i+j}} \geq \varphi^2 + (2\varphi)^j \eps.
\]
By assumption $1/r_k=P^2/(1-p)^{2^k} = \varphi^2 + \eps$ for some $\eps > 0$ and, as $2 \varphi >  1$, we have $r_i\to 0$ as claimed.
\end{proof}

Let $I$ be a constant which is large enough so that $r_I<\frac{1}{19}$. Then for $n>I$, we have $\mu_{n-I}\in~\cD(Q_{n-I},s_I,1-r_I)$ where $s_I>0$ and $1-r_I>\frac{18}{19}$, so by Proposition~\ref{prop:cube} there exists a constant $C$ (independent of $n$ and~$\mu_n$) such that, under $\mu_{n-I}$, there is a component in $Q_{n-I}$ containing at least $C\cdot2^{n-I}$ vertices with probability at least $1-e^{-\Omega(n^2)}$. By the construction of $\mu_{n-k},\dots,\mu_{n-I}$, the existence of a component of size $s$ in $Q_{n-I}$ under $\mu_{n-I}$ implies the existence of a component of size at least $s\cdot2^I$ in $Q_n$ under $\mu_n$. It follows that with probability at least $1-e^{-\Omega(n^2)}$, the random subgraph of $Q_n$ under $\mu_n$ has a component containing at least $C\cdot2^n$ vertices, which completes the proof of the theorem.
\end{proof}

\subsection{Proof of Proposition~\ref{prop:cube}}\label{sec:cube-prop}

In this section we will prove Proposition~\ref{prop:cube} via a series of lemmas. Throughout this section, when considering a model in $\cD(Q_n,p_v,p_e)$ for some $n\in \N$ and $p_v,p_e\in[0,1]$ we will write $G$ for the associated random subgraph of~$Q_n$, and $H$ for the random subgraph of $Q_n$ induced by the open vertices (thus $G$ is the subgraph of $H$ obtained by deleting closed edges).
For any vertex $v$ of $Q_n$, we will denote by $H_v$
the subgraph of $H$ induced by vertices with the same first coordinate as~$v$. For each $n\in\N$, order the vertices of $Q_n$ deterministically, for example using the lexicographic ordering. We will assume that subsets of $V(Q_n)$ inherit this ordering. In particular, for $v\in V(Q_n)$ we will often work with the first $\ceil{p_v^2 n^2/4}$ vertices in $N_{H_v}^2(v)$ when using the fixed global ordering, and we will denote this set by $I(v)$ provided it exists.  We make the following definitions.

\begin{definition}
Let $n\in \N$, $p_v\in (0,1]$, $p_e\in(0,1]$, and $\eps\in(0,1)$. Assume we are given a percolation model in $\cD(Q_n,p_v,p_e)$.
\begin{enumerate}[label=(\roman*)]
 \item A vertex $v$ of $Q_n$ is \emph{vertex-good} if it is open and $|N^2_{H_v}(v)|\geq p_v^2n^2/4$. In this case we denote by $I(v)$ the first $\ceil{p_v^2n^2/4}$ vertices of $N^2_{H_v}(v)$.
 \item A vertex $v$ of $Q_n$ is \emph{$\eps$-edge-good} if it is vertex-good and at least $(1-\eps)(2p_e-1)p_v^2n^2/4$ vertices of $I(v)$ are joined to $v$ by a path of length 2 in~$G$.
 \item An unordered pair $\{u,v\}$ of distinct vertices of $Q_n$ is \emph{$\eps$-vertex-good} if $u$ and $v$ are vertex-good and there are at least $(1-\eps)p_v^2n^2/4$ vertex-disjoint paths in~$H$, each of length at most 15, between $I(u)$ and $I(v)$.
 \item An unordered pair $\{u,v\}$ of distinct vertices of $Q_n$ is \emph{$\eps$-edge-good} if $u$ and $v$ are vertex-good and there are at least $(1-\eps)(15p_e-14)p_v^2n^2/4$ vertex-disjoint paths in~$G$ between $I(u)$ and $I(v)$.
\end{enumerate}
\end{definition}

Observe that parts \emph{(ii)} and \emph{(iv)} of this definition are only substantive when $p_e>1/2$ and $p_e>14/15$ respectively, and note that whether or not a vertex $v$ is vertex-good or a pair of vertices $\{u,v\}$ is $\eps$-vertex-good is a property of the sites only.  Also, whether or not $v$ is vertex-good or $\eps$-edge-good depends only on the sites and bonds in the copy of $Q_{n-1}$ inside $Q_n$ defined by vertices with the same first coordinate as~$v$.

We will prove Proposition~\ref{prop:cube} by first showing that for all $\eps\in(0,1)$, with probability at least $1-e^{-\Omega(n^2)}$, every pair of distinct vertex-good vertices within distance 9 of one another form an $\eps$-edge-good pair. We then show that there exists $\eps$ such that if this event occurs, then any two $\eps$-edge-good vertices at distance at most 9 from one another are in the same component in~$G$. Next, by showing that for all $\eps\in (0,1)$, with probability at least $1-e^{-\Omega(n^2)}$, every vertex in $Q_n$ has an $\eps$-edge-good vertex within $Q_n$-distance~4, we show that there exists $\eps$ such that with probability at least $1-e^{-\Omega(n^2)}$ all $\eps$-edge-good vertices are in the same component. To complete the proof of the proposition we show that for all $\eps\in(0,1)$, with probability at least $1-e^{-\Omega(n^2)}$ a constant fraction of the vertices in $Q_n$ are $\eps$-edge-good. Parts of the arguments that follow, particularly in the proofs of Lemma~\ref{lem:vertex-good-pair} and Lemma~\ref{lem:good-then-connected}, are based on arguments of McDiarmid, the fourth author, and Withers in~\cite{mcdiarmid2021component}.

\begin{lemma}\label{lem:vertex-good}
 For all\/ $p_v\in (0,1]$, $p_e\in(0,1]$, and $n\geq 12$, and all models in the class $\cD(Q_n,p_v,p_e)$, the probability that a given vertex is vertex-good is at least $p_v^3/3$.
\end{lemma}
\begin{proof}
Let $n\geq 12$ and fix a model in $\cD(Q_n,p_v,p_e)$. Fix a vertex $v\in V(Q_n)$, condition on it being open, and denote the size of its second neighbourhood in $H_v$ by $X_v=|N^2_{H_v}(v)|$. Let $u\in N^2_{Q_n}(v)$ have the same first coordinate as $v$, and let $w$ be a common neighbour of $u$ and $v$ in~$Q_n$. Since vertices of $Q_n$ are open independently, the probability that $u$ and $w$ are open given that $v$ is open is at least $p_v^2$. Thus, the probability that $u$ is in $N^2_{H_v}(v)$ is at least $p_v^2$ and $\E[X_v\mid v\text{ open}]\geq p_v^2\binom{n-1}{2}$. Hence, applying Lemma~\ref{lem:markov} and noting that $p_v^2n^2/4 \leq 2p_v^2\binom{n-1}{2}/3$ for $n \geq 12$, we have
\begin{align*}
 \Prb\big(v\text{ vertex-good}\mid v\text{ open}\big)
 &\geq \Prb\big(X_v\geq 2p_v^2\tbinom{n-1}{2}/3
  \mid v\text{ open}\big)\\
 &\geq p_v^2/3.
\end{align*}
It follows that the probability that $v$ is vertex-good is at least $p_v^3/3$ as required.
\end{proof}

\begin{lemma}\label{lem:edge-good}
 Let $p_v\in(0,1]$, $p_e\in(\frac{1}{2},1]$, and\/ $\eps\in(0,1)$ be constants. Then there exists a constant $c = c(p_v, p_e, \eps) >0$ such that for all\/ $n\ge12$ and all models in $\cD(Q_n,p_v,p_e)$, the probability that a given vertex is $\eps$-edge-good is at least $c$.
\end{lemma}
\begin{proof}
Let $n\geq 12$, fix a model in $\cD(Q_n,p_v,p_e)$ and $v\in V(Q_n)$, and condition on the event that $v$ is vertex-good.
Let $X_v$ be the number of elements of $I(v)$ which are joined to $v$ by a path of length 2 in~$G$.
Given a vertex $u \in I(v)$, let $w$ be an open neighbour of both $u$ and $v$. By a union bound,
the probability that at least one of the edges $uw$ and $wv$ is closed is at most $2(1 - p_e)$,
so the probability that both of them are open is at least $2p_e - 1$. Hence,
$\E[X_v \mid v\text{ is vertex-good}]\geq (2p_e-1)\ceil{p_v^2n^2/4}$.
Applying Lemmas~\ref{lem:markov} and~\ref{lem:vertex-good} we have
\begin{align*}
 \Prb(v\text{ $\eps$-edge-good})
 & =\Prb(X_v\geq (1-\eps)(2p_e-1)p_v^2n^2/4\mid v\text{ vertex-good})\\&\quad\cdot\Prb(v\text{ vertex-good})\\
 & \ge \eps(2p_e-1)\cdot p_v^3/3,
\end{align*}
which completes the proof of the lemma.
\end{proof}

\begin{lemma}\label{lem:vertex-good-pair}
 Let $p_v\in (0,1]$, $p_e\in(0,1]$ and\/ $\eps\in(0,1)$ be constants. Then, in any model in $\cD(Q_n,p_v,p_e)$, all pairs $\{u,v\}$ of distinct, vertex-good vertices of\/ $Q_n$ at distance at most $9$ from one another are $\eps$-vertex-good with probability at least $1-e^{-\Omega(n^2)}$.
\end{lemma}
\begin{proof}
Note that we may assume $n$ is large. Fix a model in $\cD(Q_n,p_v,p_e)$, and let $u$ and $v$ be distinct vertices of $Q_n$ at distance at most 9 from one another. We will show that the probability that $u$ and $v$ are both vertex-good, but the pair $\{u,v\}$ is not $\eps$-vertex-good is at most $e^{-\Omega(n^2)}$. Since there are at most $4^n$ pairs of vertices in $Q_n$, a union bound will complete the proof. 
    
Roughly, we will choose large sets $A \subseteq I(u)$ and $B \subseteq I(v)$, pair up the vertices in $A$ and $B$, and then to each pair $(a,b)$ associate many paths from $a$ to $b$ in $Q_n$. The paths for a given pair will be vertex-disjoint except at their endpoints so the internal vertices will all be open independently. This means it is highly likely that for at least one of these paths all the vertices will be open. The paths for different pairs will be vertex-disjoint and so there will be a vertex-open path for each pair independently. Since $A$ and $B$ are large, a Chernoff bound will complete the proof.
    
We identify $Q_n$ with $\mathcal{P}(n)$ in the natural way and assume without loss of generality that $u=\emptyset$ and $v=[d]$, where $d\leq 9$ is the distance between $u$ and~$v$. Start by revealing the states of all vertices at distance (in $Q_n$) at most 2 from $u$ or~$v$. If $u$ or $v$ is not vertex-good, then there is nothing to prove and we are done.
Otherwise we have large sets $I(u)$ and $I(v)$ of
vertices that are at distance 2 in $H$ from $u$ and $v$ respectively. Let $N=N(p_v,\eps)$ be a large constant, and construct sets $A\subseteq I(u)$ and $B\subseteq I(v)$ as follows.
\begin{enumerate}
 \item Remove from $I(u)$ or $I(v)$ any sets which differ from $u$ or $v$ in any of the elements $1,\dots,13N+9$. Let the new sets be $A_0$ and $B_0$ respectively.
 \item Arbitrarily delete elements from the larger of $A_0$ and $B_0$ until they are of equal size to obtain $A$ and $B$.
\end{enumerate}
We first show that $A$ and $B$ are large. As elements of $I(u)$ or $I(v)$ differ from $u$ and $v$ in exactly two
elements, the first step above removes at most $(13N+9)n$ elements from each of $I(u)$ and $I(v)$. As
$|I(u)|=|I(v)|=\ceil{p_v^2n^2/4}$, we have  $|A|=|B|=\min\{|A_0|,|B_0|\}\ge (1-\eps/2)p_v^2n^2/4$ for large enough~$n$.

Arbitrarily pair up elements of $A$ and $B$ to obtain $(a_1,v\cup b_1),\dots,(a_\ell,v \cup b_\ell)$ where $\ell=|A|=|B|$ and $a_i\in A$ and $v\cup b_i\in B$ for $i=1,\dots,\ell$. Here each $a_i$ and $b_i$ is a pair of elements in $\{13N+10,\dots,n\}$.
To each pair $(a,b)=(a_i,b_i)$ we associate an integer in $[13]$ which we will use to construct the paths associated to the pair. Given a pair $(a,b)$, label the elements of $a$ as $\alpha_1$ and $\alpha_2$, and the elements of $b$ as $\beta_1$ and $\beta_2$, where we assume that $\alpha_1<\alpha_2$ and $\beta_1<\beta_2$.  If $\alpha_1, \alpha_2, \beta_1$, and $\beta_2$ are all distinct, then they have six possible orderings when they are sorted into ascending order, and we associate each order with a distinct number in $[6]$ arbitrarily. If $\alpha_1 = \beta_1$ but $\alpha_2 \neq \beta_2$, then there are two possible orderings and we associate them to $7$ and $8$. Continuing in this manner, we can associate a unique integer in $[13]$ to each of the 13 possible orderings. 
    
Given a pair $(a,b)$ and $j \in [N]$, define a path $P_{a,j}$ from $a$ to $v\cup b$ as follows. Let $k \in [13]$ be the integer associated to $(a,b)$ as above, and let $x = (k-1)N+j+9$. Begin with the vertex $a$ followed by
$a\cup \{x\}$, $a\cup \{\beta_1, x\}$, and $a \cup b \cup \{x\}$, ignoring duplicate vertices. From here we add the elements of $v=[d]$ one by one in increasing order until we reach  $a\cup b\cup v\cup \{x\}$. Finally, we add the vertices $b\cup  v\cup \{\alpha_2, x\}$, $b\cup v\cup \{x\}$, and $b\cup v$. Note that the path contains at most 16 vertices, with the exact number depending on $|a \cap b|$ and~$d$.
    
\begin{claim}\label{claim:disjoint-paths}
 Let $a, a' \in A$ and $j, j' \in [N]$ be such that $a \neq a'$. Then the paths $P_{a,j}$ and $P_{a', j'}$ are vertex-disjoint, and if $j \neq j'$, then the paths $P_{a,j}$ and $P_{a,j'}$ are vertex-disjoint except at their endpoints.
\end{claim}
\begin{proof}
Let $v\cup b\in B$ be the vertex paired with $a$, and let $k\in[13]$ be the integer associated to the pair $(a,b)$. We start with the second part of the claim. Every vertex in the path $P_{a,j}$ except the two end vertices contains a unique element $x = (k-1)N+j+9$ in $[10,13N+9]$, and so $P_{a,j}$ and $P_{a,j'}$ can only share an internal vertex if $(k-1)N+j+9 = (k-1)N+j'+9$, which implies $j = j'$. Moreover, it cannot be the case that an internal vertex of $P_{a,j}$ is an endpoint of $P_{a,j'}$ as the endpoints do not contain any element of $[10,13N+9]$.
        
Turning to the first part of the claim, let $z$ be a vertex on the path $P_{a,j}$. We will show that the set $a$ is uniquely determined. Suppose first that $z$ contains an element $x\in[10,13N + 9]$. Since $x > 9$, we have $x \notin v$ and since $x \leq 13N + 9$, we have $x \notin a \cup b$. Hence, $x = (k-1)N+j+9$ and we can read off the value of $k$. If $z\cap v=\emptyset$, then $z$ is one of $a\cup \{x\}$, $a\cup \{\beta_1, x\}$, or $a \cup b \cup \{x\}$. Using the value of $k$ and the size of $z$ we can deduce which case we are in, and further, which of the elements form the pair~$a$. Similarly, if $v\subseteq z$, then $z$ is one of $a\cup b\cup v\cup \{x\}$, $b\cup  v\cup \{\alpha_2, x\}$, or $b\cup v\cup \{x\}$,
and using the value of $k$ and the size of $z$ we can deduce which case we are in, and further, which of the elements form the pair~$b$ (which then determines~$a$). Finally, if $\emptyset\ne z\cap v\ne v$ then $z\setminus v=a\cup b\cup\{x\}$ and again we can determine where we are in the path (from $|z\cap v|$) and the value of $a$ from~$k$. Finally, we consider the case where $z$ does not contain an element $x\in[10,13N + 9]$. This means $z = a$ or $z = b \cup v$, and $a$ is clearly uniquely determined in both cases.
\end{proof}
    
For each $a \in A$, let $C_a$ be the event that all the vertices are open in at least one of the paths $P_{a,j}$, $j \in [N]$. By construction, every vertex on the paths except for the endpoints are at distance at least 3 from $u$ and~$v$, and they are still open independently with probability at least~$p_v$. The paths contain at most 14 internal vertices and hence each path is open with probability at least $p_v^{14}$. As the paths are disjoint except at their endpoints, the vertices in each path are open independently and, by choosing $N$ large enough relative to $\eps$ and~$p_v$, the event $C_a$ has probability at least $1 - \eps/2$ for all~$a$.
    
By Claim~\ref{claim:disjoint-paths}, for distinct $a,a'\in A$ the events $C_a$ and $C_{a'}$ depend on disjoint sets of vertices and are independent. By a Chernoff bound, $C_a$ holds for at least $(1-\eps)p_v^2n^2/4$ values of $a$ with probability at least $1-e^{-\Omega(n^2)}$. There are at most $4^n$ pairs of vertices in $Q_n$, so by a union bound, every pair $\{u,v\}$ of distinct, vertex-good vertices in $Q_n$ at distance at most 9 from one another is $\eps$-vertex-good with probability at least $1-e^{-\Omega(n^2)}$, as required.
\end{proof}

\begin{lemma}\label{lem:edge-good-pair}
 Let $p_v\in(0,1]$, $p_e\in(\frac{14}{15},1]$, and\/ $\eps \in (0,1)$ be constants. Then, under every model in $\cD(Q_n,p_v,p_e)$, every pair of distinct, vertex-good vertices of\/ $Q_n$ at distance at most $9$ from one another form an $\eps$-edge-good pair with probability at least $1-e^{-\Omega(n^2)}$.
\end{lemma}
\begin{proof}
Let $n\in\N$ and fix a model in $\cD(Q_n,p_v,p_e)$. Reveal the graph $H$ (i.e., the states of the vertices, but not the states of the edges) and suppose that every pair of distinct vertex-good vertices at distance at most $9$ from one another form an $(\eps/2)$-vertex-good pair. By Lemma~\ref{lem:vertex-good-pair}, the probability that this does not hold is at most $e^{-\Omega(n^2)}$. Let $u$ and $v$ be distinct vertex-good vertices in $H$ at distance at most 9 from one another in~$Q_n$. We will show that the probability that the pair $\{u,v\}$ is not $\eps$-edge-good (given the graph $H$) is at most $e^{-\Omega(n^2)}$, then since there are at most $4^n$ choices for $u$ and~$v$, the probability that there are two distinct vertex-good vertices which do not form an $\eps$-edge-good pair is also at most $e^{-\Omega(n^2)}$.
    
By our assumption on $H$, the pair $\{u,v\}$ is an $(\eps/2)$-vertex-good pair so there is a set of $\ceil{(1-\eps/2)p_v^2n^2/4}$ vertex-disjoint paths in~$H$, each of length at most~15, between $I(u)$ and $I(v)$. Since the paths are vertex-disjoint and we have only conditioned on the realisation of~$H$, it follows from the fact that the model is in $\cD(Q_n,p_v,p_e)$ that each path is open independently. Further, the probability that a given path is open is at least $15p_e-14$, and applying a Chernoff bound shows that $\{u,v\}$ is an $\eps$-edge-good pair with probability at least $1-e^{-\Omega(n^2)}$, as required.
\end{proof}

In the proof of the next lemma we will make use of the following theorem of Wilson which gives the maximum number of edge-disjoint copies of $K_4$ that can be packed into $K_n$ (see also~\cite{brouwer1979optimal}).
\begin{theorem}[\cite{wilson1970construction}]\label{thm:wilson}
 If\/ $n$ is sufficiently large, then the maximum cardinality of a set of\/ $4$-subsets of $[n]$ which pairwise intersect in at most one element is $\floor{\frac{n}{4}\floor{\frac{n-1}{3}}}-1$ if\/ $n\equiv7,10\bmod{12}$ and\/ $\floor{\frac{n}{4}\floor{\frac{n-1}{3}}}$ otherwise.
\end{theorem}

In fact, it will be sufficient for our purposes that this maximum cardinality is $\Omega(n^2)$, and it is not difficult to construct suitable sets of this size.

\begin{lemma}\label{lem:good-then-connected}
 Let $p_v\in(0,1]$ and\/ $p_e\in(\frac{18}{19},1]$ be constants. Then there exists a constant $\eps\in(0,1)$ such that, under every model in $\cD(Q_n,p_e,p_v)$, the probability that all\/ $\eps$-edge-good vertices in $Q_n$ are in the same component in $G$ is at least $1-e^{-\Omega(n^2)}$.
\end{lemma}
\begin{proof}
Let $\eps\in(0,1)$ be a constant small enough that $(1-\eps)(19p_e-16)>2$. Let $n\in\N$ and fix a model in $\cD(Q_n,p_v,p_e)$. Note that we may assume $n$ is large in terms of $p_v$, $p_e$, and~$\eps$. We will start by showing that if $u$ and $v$ are distinct vertices of $Q_n$ at distance at most 9 from one another such that $u$ and $v$ are $\eps$-edge-good vertices and $\{u,v\}$ is an $\eps$-edge-good pair, then $u$ and $v$ are in the same component in~$G$.

Indeed, let $u$ and $v$ be such vertices, and note that they are also both vertex-good. As $\{u,v\}$ is $\eps$-edge-good, there is a set $S$ of at least $(1-\eps)(15p_e-14)p_v^2n^2/4$ vertex-disjoint paths from $I(u)$ to $I(v)$ in $G$. Let $T_1$ be the set of pairs $(a,b)$ such that $a$ and $b$ are the endpoints of a path in $S$, where $a\in I(u)$ and $b\in I(v)$. Arbitrarily pair up the remaining elements of $I(u)$ and $I(v)$ and let the set of all these pairs be $T$, so that $|T|=\ceil{p_v^2n^2/4}$.

Let $T_2$ be the set of pairs in $T$ whose first entry has a path to $u$ in $G$ of length 2, and let $T_3$ be the set of pairs whose second entry has a path to $v$ in $G$ of length 2. Then, since $u$ and $v$ are $\eps$-edge-good vertices, we have $|T_2|,|T_3|\geq (1-\eps)(2p_e-1)p_v^2n^2/4$. Thus,
\begin{align*}
 |T_1|+|T_2|+|T_3|& \geq (1-\eps)[2(2p_e-1)+(15p_e-14)]p_v^2n^2/4\\
 & = (1-\eps)(19p_e-16)p_v^2n^2/4.
\end{align*}
By our choice of $\eps$, we have $(1-\eps)(19p_e-16)>2$, so if $n$ is large enough in terms of $p_v$, $p_e$, and $\eps$, then there exists a pair $(a,b)\in T_1\cap T_2\cap T_3$. It follows that $u$ and $v$ are in the same component in $G$ since there is a walk from one to the other via this $a$ and~$b$.

We now prove the following claim.
\begin{claim}\label{claim:good-within-4}
 With probability at least $1-e^{-\Omega(n^2)}$, every vertex in $Q_n$ has an $\eps$-edge-good vertex within distance 4.
\end{claim}
\begin{proof}
First note that any set of vertices in $Q_n$ which are pairwise at distance at least 5 from one another are $\eps$-edge-good independently. Fix a vertex $v$ in $Q_n$ and assume without loss of generality that $v=\textbf{0}$. By Theorem~\ref{thm:wilson} there exists a set $S$ of vertices at distance 4 from $v$ in $Q_n$ which are pairwise at distance at least 6 in $Q_n$ with $|S|=\Omega(n^2)$. The vertices in $S$ are $\eps$-edge-good independently, and by Lemma~\ref{lem:edge-good} (noting that we may assume that $n$ is large), there exists a constant $c>0$ such that each of them is $\eps$-edge-good with probability at least~$c$. Hence, the probability that at least one of the vertices in $S$ is $\eps$-edge-good is at least $1-e^{-\Omega(n^2)}$, from which the claim follows by a union bound.
\end{proof}
    
Now suppose that every pair of distinct, vertex-good vertices of $Q_n$ at distance at most 9 from one another form an $\eps$-edge-good pair and every vertex in $Q_n$ has an $\eps$-edge-good vertex within distance 4. We will show that under these assumptions, if $u$ and $v$ are distinct $\eps$-edge-good vertices of $Q_n$, then they are in the same component of $G$. Since our assumptions hold with probability $1-e^{-\Omega(n^2)}$, this will complete the proof.
    
Fix a path between $u$ and $v$ in $Q_n$, say $x_0x_1\dots x_k$ where $x_0=u$ and $x_k=v$. If $k \leq 9$, then $u$ and $v$ are two $\eps$-edge-good, and therefore vertex-good, vertices within distance~9. Hence, they form an $\eps$-edge-good pair and we are done by the first part of the proof. If instead $k > 9$, then each of $x_5,\dots,x_{k-5}$ are within distance 4 of $\eps$-edge-good vertices, say $y_5,\dots,y_{k-5}$ respectively. Then $x_0$ and $y_5$ are both vertex-good vertices, and since they are within distance 9 of each other, they form an $\eps$-edge-good pair. By the first part of the proof, they are in the same component in $G$. By similar logic, so are $x_k$ and $y_{k-5}$, and $y_i$ and $y_{i+1}$ for each $i\in\{5,\dots,k-6\}$. Hence, $u$ and $v$ are in the same component and we are done.
\end{proof}

We are now ready to prove Proposition~\ref{prop:cube}.

\begin{proof}[Proof of Proposition~\ref{prop:cube}]
By Lemma~\ref{lem:good-then-connected} it is sufficient to show that for all constants $\eps\in(0,1)$ there exists a positive constant $C$ such that for all $n\in\N$ and models in $\cD(Q_n,p_e,p_v)$, there are at least $C\cdot2^n$ $\eps$-edge-good vertices in $Q_n$ with probability at least $1-e^{-\Omega(n^2)}$. Fix an $\eps \in (0,1)$, let $n \geq 6$, and consider a model in $\cD(Q_n,p_e,p_v)$.

We observed in the proof of Claim~\ref{claim:good-within-4} that any set of vertices which are pairwise at distance at least 5 in $Q_n$ are $\eps$-edge-good independently of one another, and by Lemma~\ref{lem:edge-good} (noting that we may assume that $n$ is large) there exists a constant $c>0$ such that the probability that a given vertex of $Q_n$ is $\eps$-edge-good is at least~$c$. Consider the fifth power $Q_n^{(5)}$ of $Q_n$, that is the graph with vertex set $V(Q_n)$ and edges between vertices which are at distance at most 5 from one another in~$Q_n$. This graph is $\Delta$-regular for some $\Delta=O(n^5)$, so by Brooks' theorem we can properly vertex-colour $Q_n^{(5)}$ with $\Delta$ colours.

Each colour class consists of vertices which are pairwise at distance at least 5 from one another in~$Q_n$, so they are $\eps$-edge-good independently of one another. Hence, by a Chernoff bound, there exists a constant $C' = C'(p_v, p_e, \eps) >0$ such that if $D$ is a colour class with $|D|\geq n^2$, then at least $C'|D|$ of the vertices in $D$ are $\eps$-edge-good with probability at least $1-e^{-\Omega(n^2)}$. Since $\Delta=O(n^5)$ it follows by a union bound that, with probability at least $1-e^{-\Omega(n^2)}$, in every colour class of size at least $n^2$ at least a $C'$ proportion of the vertices are $\eps$-edge-good. The total number of vertices in colour classes of size less than $n^2$ is at most $\Delta \cdot n^2=O(n^7)$, so there exists a constant $C>0$ such that with probability at least $1-e^{-\Omega(n^2)}$, at least a $C$ proportion of the vertices of $Q_n$ are $\eps$-edge-good, as required.
\end{proof}

\section{Percolation in \texorpdfstring{$\Z^n$}{Z\textasciicircum{}n}}\label{sec:lattice}

In this section we build on the methods and results of Section~\ref{sec:cube} to prove Theorem~\ref{thm:lattice}. We also prove Lemma~\ref{lem:q6-p}, which we need in order the apply the theorem, in Section~\ref{sec:q3-p}. As was the case with Theorem~\ref{thm:cube}, our proof of Theorem~\ref{thm:lattice} rests on a renormalisation argument which reduces the problem to one concerning percolation models of the form given
in Definition~\ref{def:renormalised}. We now state the analogue of Proposition~\ref{prop:cube} for this setting, before proceeding to prove Theorem~\ref{thm:lattice} from the proposition by mimicking the proof of Theorem~\ref{thm:cube} from Proposition~\ref{prop:cube}. Following this we will prove the proposition.

\begin{proposition}\label{prop:lattice}
 Let $p_v\in (0,1]$ and\/ $p_e\in (\frac{18}{19},1]$ be constants. Then, for large enough~$n$, all models in $\cD(\Z^2\times\{0,1\}^{n-2},p_v,p_e)$ percolate.
\end{proposition}

\begin{proof}[Proof of Theorem~\ref{thm:lattice}]
    To prove the theorem it is sufficient to show that if $n$ is large, then every 1-independent bond percolation model on $\Z^2\times\{0,1\}^{n-2}$ in which each edge is open with probability $p$ percolates. Let $n$ be large and let $\mu_{n}$ be such a model on $\Z^2\times\{0,1\}^{n-2}$. Define a percolation model $\mu_{n-k}$ on $\Z^2\times\{0,1\}^{n-k-2}$ by defining a vertex $v$ to be open if the hypercube $H_v = \{v\concat a:a\in\{0,1\}^k\}$ is connected under $\mu_n$, and defining an edge of $\Z^2\times\{0,1\}^{n-k-2}$ between two open vertices $u$ and $v$ to be open if at least one of the $2^k$ edges in $\Z^2\times\{0,1\}^{n-2}$ between $H_u$ and $H_v$ is open under $\mu_n$. For $i\in \{k,\dots,n-3\}$, recursively define a model $\mu_{n-i-1}$ on $\Z^2\times\{0,1\}^{n-i-3}$ by defining a vertex $v$ to be open if the edge between $v\concat0$ and $v\concat1$ is open in $\Z^2\times\{0,1\}^{n-i-2}$ under $\mu_{n-i}$, and setting an edge $uv$ to be open if the edge between $u\concat0$ and $v\concat0$ or the edge between $u\concat1$ and $v\concat1$ is open in $\Z^2\times\{0,1\}^{n-i-2}$ under $\mu_{n-i}$.

    Let $q = 1-p$, and define $q_i$, $s_i$, and $r_i$ for $i\geq k$ as in the proof of Theorem~\ref{thm:cube}. Then we again have that $r_i\rightarrow 0$ as $i\rightarrow \infty$ and, for all $k\leq i\leq n-2$, we have $s_i>0$ and $\mu_{n-i}\in\cD(\Z^2\times\{0,1\}^{n-i-2}, s_i, 1-r_i)$. Thus, we can take $I$ to be a constant large enough that $r_I<\frac{1}{19}$, so that for $n\geq I+2$ we have $\mu_{n-I}\in\cD(\Z^2\times\{0,1\}^{n-I-2}, s_I, 1-r_I)$ where $s_I>0$ and $1-r_I>\frac{18}{19}$, and Proposition~\ref{prop:lattice} applies. Hence, $\mu_{n-I}$ percolates. By the construction of $\mu_{n-I}$ from~$\mu_n$, the existence of an infinite component in $\Z^2\times\{0,1\}^{n-I-2}$ under $\mu_{n-I}$, implies the existence of an infinite component in $\Z^2\times\{0,1\}^{n-2}$ under $\mu_n$, and hence $\mu_n$ percolates, which completes the proof of the theorem.
\end{proof}

We now turn our attention to Proposition~\ref{prop:lattice}.

\begin{proof}[Proof of Proposition~\ref{prop:lattice}]
To simplify the notation, we will prove the equivalent statement that for large enough~$n$, all models in $\cD(\Z^2\times\{0,1\}^{n-1}, p_v, p_e)$ percolate. To begin, apply Lemma~\ref{lem:good-then-connected} to find $\eps\in(0,1)$ such that for each $n\in\N$, under every model in $\cD(Q_n,p_e,p_v)$, with probability at least $1-e^{-\Omega(n^2)}$ all $\eps$-edge-good vertices in $Q_n$ are in the same component. Now let $n$ be large and consider $\mu\in \cD(\Z^2\times\{0,1\}^{n-1}, p_v, p_e)$. Let $G$ be the random subgraph of $\Z^2\times\{0,1\}^{n-1}$ associated with~$\mu$. To each edge $e\in E(\Z^2)$, we associate the natural copy of $Q_n$ in $\Z^2\times\{0,1\}^{n-1}$, which we denote by~$C_e$. The subgraph of $G$ induced on $V(C_e)$ follows a model in $\cD(Q_n,p_v,p_e)$, so we can define a bond percolation model $\nu=\nu(\mu)$ on $\Z^2$ coupled to $\mu$ by declaring $e\in E(\Z^2)$ to be open if the following two conditions hold:
\begin{enumerate}[label=(\alph*)]
 \item in the subgraph of $G$ induced on $V(C_e)$, all $\eps$-edge-good vertices of $C_e$ are in the same component; and
 \item if we partition $V(C_e)$ into two equal parts according to the first two coordinates of the vertices, then there are $\eps$-edge-good vertices of $C_e$ in both halves.
\end{enumerate}

Since $\mu\in\cD(\Z^2\times\{0,1\}^{n-1}, p_v, p_e)$, it is clear that $\nu$ is a 1-independent bond percolation model on~$\Z^2$. By our choice of~$\eps$, for each edge $e$ condition (a) holds with probability at least $1-e^{-\Omega(n^2)}$. For condition (b), recall that sets of vertices of $Q_n$ which are pairwise at distance at least 5 are $\eps$-edge-good independently. It is straightforward to see that there exists such a subset of $V(C_e)$ in which $\Omega(n^2)$ vertices are taken from each half of the cube, so applying Lemma~\ref{lem:edge-good} and using a Chernoff bound on each half we find that for each edge $e$ condition (b) holds with probability at least $1-e^{-\Omega(n^2)}$.
Thus $\nu$ is a 1-independent bond percolation model on $\Z^2$ in which each edge is open with probability at least $1-e^{-\Omega(n^2)}$. Since $\pmax(\Z^2)<1$, for example by Theorem~\ref{thm:z2-upper-bound} or as shown in~\cite{balister2005continuum} and in~\cite{liggett1997domination}, we see that $\nu$ percolates if $n$ is large enough in terms of $p_v$ and $p_e$.

It remains to show that percolation of $\nu$ implies percolation of~$\mu$. We claim that for any collection of connected, open edges under~$\nu$, the $\eps$-edge-good vertices in the associated hypercubes are in the same component under~$\mu$. An infinite component under $\nu$ contains infinitely many disjoint edges and the hypercubes of each of these must contain two $\eps$-edge-good vertices, so this immediately implies that $\mu$ percolates when $\nu$ percolates.
To prove the claim, suppose that $v\in V(\Z^2\times\{0,1\}^{n-1})$ is $\eps$-edge-good when considered as a vertex in $C_e$ for some $e\in E(\Z^2)$. Let $f$ be another edge of $\Z^2$ such that $v\in V(C_f)$ (so that $e$ and $f$ have a common vertex in $\Z^2$). Since the definition of a vertex in a copy of $Q_n$ being $\eps$-edge-good depends only on the $Q_{n-1}$ subgraph with the same first coordinate as that vertex, it follows that $v$ is also $\eps$-edge-good when considered as a vertex of $C_f$. Now if $e$ and $f$ are distinct open edges of $\Z^2$ with a common vertex, then by condition (b) applied to $e$ or $f$ there exists an $\eps$-edge-good vertex in $V(C_e) \cap V(C_f)$. By condition (a) applied to $e$ and $f$, all the $\eps$-edge-good vertices in $V(C_e)\cup V(C_f)$ are in the same component in~$G$, and the claim follows immediately.
\end{proof}

\subsection{Proof of Lemma~\ref{lem:q6-p}}\label{sec:q3-p}

We conclude this section with the proof of Lemma~\ref{lem:q6-p}, which allows us to extract concrete bounds from Theorems~\ref{thm:lattice} and~\ref{thm:cube}.

\begin{proof}[Proof of Lemma~\ref{lem:q6-p}]
Given a model in $\cD_{\geq 0.5847}(Q_6)$, we start by constructing a model in $\cD(Q_3,p_v,p_e)$ for some $p_v,p_e\in(0,1]$ by applying the first step (with $k=3$) of the renormalisation process used in the proof of Theorem~\ref{thm:cube}. That is, we declare a vertex $v$ of $Q_3$ to be open if the cube $\{v\concat a:a\in\{0,1\}^3\}$ is connected under the model on $Q_6$, and an edge of $Q_3$ between two open vertices to be open if at least one of the eight edges between the two cubes corresponding to its endpoints is open. If $P_0>0$ is a constant such that in any model in $\cD_{\geq 0.5847}(Q_3)$ the probability that the graph is connected is at least $P_0$, then by Claim~\ref{claim:renorm} the renormalised model on $Q_3$ is in $\cD(Q_3,P_0,1-0.4153^8/P_0^2)$, where $0.4153=1-0.5847$.
    
It is clear that this renormalisation has the property that if the random subgraph associated to the renormalised model forms a connected spanning subgraph of $Q_3$, then the random subgraph of $Q_6$ associated to the original model is connected. Moreover, if we condition on all sites being open in the renormalised model, then the bonds follow a model in $\cD_{\geq1-0.4153^8/P_0^2}(Q_3)$. Thus if $P_1>0$ is a constant such that for any model in this class, the probability that the random subgraph of $Q_3$ is connected is at least $P_1$, then the probability that $Q_6$ is connected under the original model is at least $P_0^8P_1$.
    
In order to obtain suitable values for $P_0$ and $P_1$, we construct a linear program that is satisfied by any 1-independent model on $Q_3$ (with edge probability exactly $p$) by removing the non-linear conditions.
More specifically, for each subset $S$ of the edges of $Q_3$, let $x_S$ denote the probability that the set of open edges is exactly $S$, and $y_S$ denote the probability that $S$ is a subset of the set of open edges.
Let $\cC$ be the collection of all subsets $S$ of $E(Q_3)$ that form a connected spanning subgraph of $Q_3$. We consider the following linear programming problem.
\[\begin{array}{lr@{\ }c@{\ }l}
 \text{Minimise: }   &\sum_{S\in\cC}x_S\kern-0.9cm \\[6pt]
 \text{Subject to: }
 &x_S            &\ge& 0,\\
 &y_S            &=&\sum_{T\supseteq S}x_T,\\
 &y_{S\cup\{e\}} &=&p\cdot y_S,\\
 &y_{\emptyset}  &=&1,
\end{array}\]
where $S$ runs over all subsets of $E(Q_3)$ and
$e$ runs over all edges that are vertex-disjoint from
all edges in $S$.

It is clear that the above conditions hold in any 1-independent model with edge probability~$p$. Indeed, the only conditions missing are the non-linear constraints $y_{S\cup T}=y_Sy_T$ when $S$ and $T$ are sets of edges
sharing no vertex and $|S|,|T|\ge2$. Thus the solution
to the linear programming problem gives a lower
bound on the minimum probability that the open
edges in any 1-independent bond percolation model on $Q_3$ in which edges are open with probability $p$
form a connected spanning subgraph of~$Q_3$. As the existence of a spanning connected open subgraph is an increasing event, this bound also holds when the edges are open with probability at least~$p$ as we can independently delete edges so as to ensure edges are open with probability exactly~$p$.
    
Running the above LP problem using the \texttt{Gurobi}
optimisation package and $p=0.5847$ gave a lower
bound\footnote{We also ran the dual programs and checked that the dual solutions provided by \texttt{Gurobi} were feasible.} on $P_0$ of
$0.0463$. Running the LP problem again with $p=0.5872<1-0.4153^8/0.0463^2$ gave a lower bound on $P_1$ of $0.0497$. Hence we may take $P=10^{-12}<0.0497\cdot(0.0463)^8$ as a lower bound on the probability of $Q_6$ being connected under any model in $\cD_{\geq 0.5847}(Q_6)$, and $P>9.93\times 10^{-13}>0.4153^{32}\cdot\varphi$ as required.
\end{proof}

\section{Lower bounds on \texorpdfstring{$\pmax(\Z^2)$}{pmax(Z\textasciicircum{}2)}}\label{sec:z2-lower-bounds}

In this section we detail the proof of Theorem~\ref{thm:z2-lower-bound} and justify Result~\ref{res:high-conf}, starting with the former.

\begin{proof}[Proof of Theorem~\ref{thm:z2-lower-bound}.]
Let $p>\psite=\psite(\Z^2)$ and, to each vertex $(i,j)\in\Z^2$ with $i+j\equiv 0\bmod2$, assign independent random variables
$X_{i,j}\in\{\text{A},\text{U},\text{D},\text{L},\text{R}\}$ which take the value `$\text{A}$' with probability $1-p$ and
each of the other four values with probability~$p/4$.
We add edges according to the following rules, and leave all other edges closed.
\begin{itemize}[noitemsep]
\item If $X_{i,j}=\text{U}$, we add the edge from $(i,j)$ to $(i,j+1)$.
\item If $X_{i,j}=\text{D}$, we add the edge from $(i,j)$ to $(i,j-1)$.
\item If $X_{i,j}=\text{L}$, we add the edge from $(i,j)$ to $(i-1,j)$.
\item If $X_{i,j}=\text{R}$, we add the edge from $(i,j)$ to $(i+1,j)$.
\item If $X_{i,j}=\text{A}$, we add all four edges incident to $(i,j)$.
\end{itemize}
This clearly gives a 1-independent bond percolation model on $\Z^2$ and, moreover, any infinite open
path must, on alternate internal vertices, pass through vertices in the even sublattice $\{(i,j):i+j\equiv0\bmod 2\}$
at vertices where $X_{i,j}=\text{A}$ (as $(i,j)$ has degree 1 when $X_{i,j}\ne\text{A}$).
Thus, we must have an infinite component in the `diagonally connected'
even sublattice, where each vertex is adjacent to its eight nearest neighbours. It is well known that site percolation on this
lattice satisfies a `duality' condition with site percolation on the usual $\Z^2$ lattice and that the percolation threshold
is $1-\psite$ (see, for example,~\cite{bollobas2006percolation}). As $p>\psite$ and $\Prb(X_{i,j}=\text{A})=1-p$, we almost
surely do not have such an infinite component, and hence this model does not percolate.

In this model, each edge is open with probability $\frac{p}{4}+(1-p)=1-\frac{3}{4}p$, so
\[
 \pmax(\Z^2)\ge 1-\tfrac{3}{4}\psite.
\]
Now, as noted in the introduction, Day, Falgas-Ravry, and Hancock~\cite{day2020long} showed that
\[
 \pmax(\Z^2)\ge \psite^2+\tfrac{1}{2}(1-\psite).
\]
Hence, \emph{independently} of any assumption on the value of $\psite$, we have
\[
 \pmax(\Z^2)\ge \inf_{x\in[0,1]}\max\big\{1-\tfrac34x,x^2+\tfrac12(1-x)\big\}
 =\tfrac{1}{32}(35-3\sqrt{33})\approx 0.555197.\qedhere
\]
\end{proof}

Note that if we use the conjectured value $\psite\approx 0.592746$ we get
the slightly better bound $\pmax(\Z^2)\ge 0.555440$. However, an even better
non-rigorous bound is given by the following model.

Let each site $v$ be given a state $X_v\in \{\text{U},\text{D},\text{L},\text{R}\}$ independently at random with
$\Prb(X_v=\text{U})=\Prb(X_v=\text{R})=\frac{1}{2}\theta$, and $\Prb(X_v=\text{D})=\Prb(X_v=\text{L})=\frac{1}{2}(1-\theta)$,
where $\theta\in(0,1)$. We join neighbouring sites $u = (i,j)$ and $v$ if {\em any} of the following hold.
\begin{itemize}[noitemsep]
 \item $X_u=X_v$;
 \item $X_u=\text{U}$ and $v = (i,j+1)$ is above $u$;
 \item $X_u=\text{D}$ and $v = (i,j-1)$ is below $u$;
 \item $X_u=\text{L}$ and $v = (i-1,j)$ is to the left of $u$;
 \item $X_u=\text{R}$ and $v = (i+1,j)$ is to the right of $u$.
\end{itemize}
In other words, sites are joined if their states are equal, but a site in state $d$ also forces
an edge to its neighbour in direction $d$ regardless of the neighbour's state.

If $v$ is to the right of $u$ we have
\begin{align*}
 \Prb(uv\text{ is open})
 &=\Prb(X_u=X_v=\text{U})+\Prb(X_u=X_v=\text{D})+\Prb(X_u=\text{R or }X_v=\text{L})\\
 &=\tfrac{1}{4}\theta^2+\tfrac{1}{4}(1-\theta)^2+\tfrac{1}{2}\theta+\tfrac{1}{2}(1-\theta)-\tfrac{1}{4}\theta(1-\theta)=\tfrac{3}{4}(\theta^2-\theta+1).
\end{align*}
A similar calculation holds for vertical bonds and so every edge is present with probability $p=\frac{3}{4}(\theta^2-\theta+1)$. This is clearly a 1-independent model, and numerical simulations suggest the threshold for percolation (above which the model percolates) is at $p\approx0.592119$.

We provide a high confidence result that this model does not percolate for $p$ just above $0.5921$, implying that $\pmax(\Z^2)>0.5921$. See, for example,~\cite{balister2005continuum} or~\cite{riordan2007rigorous} for more examples of this method.
We use a renormalisation argument, constructing a new 1-independent model on $\Z^2$ from the model described above by identifying renormalised sites $(x,y)$ with $N\times N$ blocks of sites $(Nx,Ny)+\{0,\dots,N-1\}^2$ in the original model. A renormalised bond $uv$, which corresponds to a $2N\times N$ (or $N\times 2N$) rectangle, will be open if some event $E_{uv}$ holds, and this event will depend only on the sites and bonds within the corresponding rectangle to ensure 1-independence of the renormalised model. These events will be chosen so that if there is an open cycle $C$ enclosing a point $v$ in the renormalised model, then in the original model the open component containing the point $v$ is contained within the bounded region enclosed by the blocks corresponding to $C$.

If $\Prb(E_{uv})>0.8457$ for all edges $uv$, then
by Corollary~\ref{cor:z2strong} below there are almost surely cycles enclosing any bounded region in the plane. Hence, no point is in an infinite component in the original model, and so the original model does not percolate. In principle, it is possible to calculate $\Prb(E_{uv})$
exactly, but this is usually impractical unless $N$
is very small. Instead, we use Monte Carlo simulations
to test the hypothesis that $\Prb(E_{uv})>0.8457$.
We run $T$ trials, constructing a pseudorandom instance
of the states in a $2N\times N$ rectangle then determining
whether or not $E_{uv}$ holds. If $\Prb(E_{uv})\le 0.8457$, then with probability at most $p=\Prb(\mathrm{Bin}(T,0.8457)\ge k)$ will we have
$k$ or more successful trials. Thus, if we obtain $k$
successful trials and $p$ is small, then we can say with high confidence that $\Prb(E_{uv})>0.8457$.
Other than a very unfortunate coincidence occurring
in the simulation, there are two possible reasons this may fail. One is that the pseudorandom number
generator we use may not be sufficiently random, and the other
is possible errors in the computer software. To mitigate these errors, we performed two different
experiments with different software, different choices of $E_{uv}$,
and different pseudorandom number generators.

Before giving the details of the first experiment, we recall that the dual graph of the $\Z^2$ lattice has sites corresponding to the
square faces of the lattice, and bonds
joining these faces are open exactly when the unique bond of
the original lattice crossing them is closed. In the first experiment, which closely follows the
method in~\cite{riordan2007rigorous}, $E_{uv}$ is the event that
\begin{itemize}
 \item there is a unique largest connected component $C_u$ in the $(N-1)\times (N-1)$ dual graph in the block corresponding to $u$;
 \item there is a unique largest connected component $C_v$ in the $(N-1)\times (N-1)$ dual graph in the block corresponding to $v$; and
 \item $C_u$ and $C_v$ both lie in the same connected component of the $(2N-1)\times (N-1)$ dual graph in the $2N\times N$ rectangle corresponding to the bond $uv$.
\end{itemize}
Note that since we want $E_{uv}$ to depend only on the model inside
the $2N\times N$ rectangle corresponding to $uv$, we are
restricted to only considering dual bonds in a $(2N-1)\times (N-1)$ rectangle. It is straightforward to see that these events $E_{uv}$ have the required property that the existence of an open cycle enclosing a region in the renormalised model implies the existence of an open cycle enclosing the region in the dual graph of the original model, and hence any open component of the original model meeting this region must be finite.

Apart from the choice of random
number generator, the algorithm for finding the
component structure in the dual graph is the same as
in~\cite{riordan2007rigorous}. In particular, states are revealed
column by column, and the component structure is updated accordingly.
Thus only states in two or three
columns are retained, reducing memory requirements
to $O(N)$ (rather than storing all the states which
would use $O(N^2)$ memory). 
We used a 16-bit version of the \texttt{arc4}
algorithm
to supply the random numbers (see, for example, \cite{rivest2016spritz} for an overview of the algorithm).\footnote{While the \texttt{arc4} algorithm is known to have some biases (see, for example, \cite{fluhrer2001statistical,mantin2005predicting}), by discarding the initial outputs and using 16-bits rather than 8-bits,
all known biases in \texttt{arc4} are sufficiently small
so as to not affect the simulations.}
Taking
$N=1{,}200{,}000$ and $T=300$ we obtained
296 successes, giving a $p$-value of less than $10^{-16}$.

The second experiment follows a method suggested (but not used) in~\cite{riordan2007rigorous}. For a horizontal edge $uv$ with $u$ to the left of $v$, $E_{uv}$ is the event that
\begin{itemize}
 \item there is no open path from the bottom to the top of the $2N\times N$ rectangle corresponding to the edge $uv$; and
 \item there is no open path from the left to the right of the $N\times N$ square corresponding to $u$.
\end{itemize}
Reflecting in the line $x=y$ gives
a corresponding definition for vertical edges.
By symmetry these have the same probability, $\Prb(E_{uv})$. Again, it is not difficult to see that these events have the required properties.

To determine if an open path crosses a square or rectangle
we use a boundary following algorithm where we
explore the boundary of the sites connected to
the bottom of the rectangle (or left side of the square). One advantage of this algorithm is that it can be much faster
as it only needs to determine the states of
a small fraction of the sites. The disadvantage is
that we now need a pseudorandom \emph{function} which can
generate consistent random data for all sites $(x,y)$
accessed in any order (rather than a sequential list
of random numbers as is more usual in a random number
generator). For this we used the 20-round \texttt{chacha}
algorithm introduced in \cite{bernstein2008chacha} to generate our random states. We again set
$N=1{,}200{,}000$ and performed 300 trials.
There were 291 successes, giving a $p$-value of less than
$10^{-11}$.

The code for both of these experiments is attached to the \href{https://arxiv.org/abs/2206.12335}{arXiv} submission.

\section{Upper bounds on \texorpdfstring{$\pmax(\Z^2)$}{pmax(Z\textasciicircum{}2)}} \label{sec:z2-upper-bound}

We now prove Theorem~\ref{thm:z2-upper-bound} which gives an improved upper bound on $\pmax(\Z^2)$. We will apply a renormalisation argument and use linear programs to lower bound the probability that each edge is present. 

\begin{proof}[Proof of Theorem~\ref{thm:z2-upper-bound}.]
We follow the basic renormalisation idea used in~\cite{balister2005continuum}. Tile $\Z^2$ with $2\times 2$ squares $S_{(i,j)}=\{2i,2i+1\}\times\{2j,2j+1\}$ and form a new 1-independent model on~$\Z^2$ by identifying
$S_{(i,j)}$ with the site $(i,j)\in\Z^2$ and joining
neighbouring sites $u,v\in\Z^2$ if a certain `good event' $E_{uv}$ occurs in the corresponding $4\times 2$ or $2\times 4$ rectangle $S_u\cup S_v$ in the original 1-independent model. The good event $E_{uv}$ will be defined to be the existence of an open component in the rectangle $S_u\cup S_v$ that has a `large' intersection with both $S_u$ and $S_v$. As long as any two `large' subsets of any $S_u$ intersect, these events will have the property that an open path from some $v$ to some $w$ in the renormalised grid will imply the existence of an open path from some site in $S_v$ to some site in $S_w$ in the original model.

Clearly containing at least three out of four sites in $S_u$ would be a sufficient condition to be `large', but we can do better. Fix for each $u\in\Z^2$ a subset $\cL_u\subseteq S_u$ of size~3, i.e., one of the subsets shown in Figure~\ref{fig:s_ij}. We define a large subset of $S_u$ to be one that intersects $\cL_u$ in at least two sites. Note that this includes any set containing three or more sites in~$S_u$, but certain 2-element subsets of $S_u$ are now also deemed large and, of course, any two large subsets of $S_u$ intersect.

We now come to the choice of $\cL_u$. These could be chosen to be all (translates of) the same fixed set, but we can improve the bound slightly by having $\cL_u$ vary, and indeed vary randomly. Hence, we shall choose $\cL_u$ randomly, independently for each~$u$, and also independently of the state of the original model. We note that under these assumptions, the renormalised model is still 1-independent as vertex-disjoint sets of renormalised edges depend on vertex-disjoint subgraphs of the original model, and on the choice of $\cL_u$ for disjoint sets of sites~$u$. 

Even the probability distributions for the choice of $\cL_u$ will vary, depending on~$u$, so our renormalised bonds will not be open with the same probabilities. Hence to inductively renormalise we will need to consider 1-independent models with differing edge probabilities for each edge. However, it will be enough for our purposes to consider models with at most two distinct edge probabilities. More precisely, define $\cD_{p,p'}$ as the set of 1-independent bond percolation models on $\Z^2$ in which each edge $uv$ is open with probability $p$ if $u$ and $v$ both lie in the same $2\times2$ square $S_w$, and $p'$ otherwise (see Figure~\ref{fig:z2-upper-bound}). Similarly we define $\cD_{\ge p,\ge p'}$ as the set of 1-independent models where the edge probabilities are at least $p$ and $p'$ respectively.

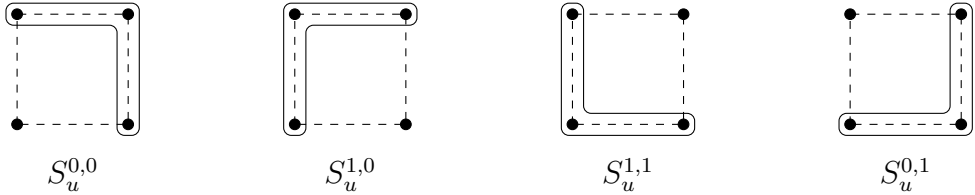
\begin{figure}[t]
\centering
\tikzmath{\gap = 1.5;}
\tikzmath{\width = 4.2 + 3*\gap;}
\begin{tikzpicture}[scale=0.8*(\linewidth/1cm)/\width]
\begin{scope}[shift={(0,0)}]
 \draw node[style=vertex] at (0,0) {};
 \draw node[style=vertex] at (0,1) {};
 \draw node[style=vertex] at (1,0) {};
 \draw node[style=vertex] at (1,1) {};
 \draw[dashed] (0,0) -- (0,1) -- (1,1) -- (1,0) -- (0,0);
 \draw[rounded corners = 0.1cm] (1.1, 0.1) -- (1.1, -0.1) -- (0.9, -0.1) -- (0.9,0.9) -- (-0.1, 0.9) -- ( -0.1, 1.1) -- (1.1, 1.1) -- (1.1, 0.1);
 \node[below=0.3cm] at (0.5, 0) {$S_u^{0,0}$};
\end{scope}
\begin{scope}[shift={(1+\gap,0)}]
 \draw node[style=vertex] at (0,0) {};
 \draw node[style=vertex] at (0,1) {};
 \draw node[style=vertex] at (1,0) {};
 \draw node[style=vertex] at (1,1) {};
 \draw[dashed] (0,0) -- (0,1) -- (1,1) -- (1,0) -- (0,0);
 \draw[rounded corners = 0.1cm] (-0.1, 0.1) -- (-0.1, 1.1) -- (1.1, 1.1) -- (1.1, 0.9) -- (0.1,0.9) -- (0.1, -0.1) -- (-0.1, -0.1) -- (-0.1, 0.1);
 \node[below=0.3cm] at (.5, 0) {$S_u^{1,0}$};
\end{scope}
\begin{scope}[shift={(2+2*\gap,0)}]
 \draw node[style=vertex] at (0,0) {};
 \draw node[style=vertex] at (0,1) {};
 \draw node[style=vertex] at (1,0) {};
 \draw node[style=vertex] at (1,1) {};
 \draw[dashed] (0,0) -- (0,1) -- (1,1) -- (1,0) -- (0,0);
 \draw[rounded corners = 0.1cm] (-0.1, 0.1) -- (-0.1, 1.1) -- (0.1, 1.1) -- (0.1,0.1) -- (1.1, 0.1) -- (1.1, -0.1) -- (-0.1, -0.1) -- (-0.1, 0.1);
 \node[below=0.3cm] at (.5, 0) {$S_u^{1,1}$};
\end{scope}
\begin{scope}[shift={(3+3*\gap,0)}]
 \draw node[style=vertex] at (0,0) {};
 \draw node[style=vertex] at (0,1) {};
 \draw node[style=vertex] at (1,0) {};
 \draw node[style=vertex] at (1,1) {};
 \draw[dashed] (0,0) -- (0,1) -- (1,1) -- (1,0) -- (0,0);
 \draw[rounded corners = 0.1cm] (0.1, -0.1) -- (-0.1, -0.1) -- (-0.1, 0.1) -- (0.9,0.1) -- (0.9, 1.1) -- (1.1, 1.1) -- (1.1, -0.1) -- (0.1, -0.1);
 \node[below=0.3cm] at (.5, 0) {$S_u^{0,1}$};
 \end{scope}
\end{tikzpicture}
\caption{The sets $S_u^{r,s}$} 
\label{fig:s_ij}
\end{figure}

We now define the precise choice of probability distribution for the $\cL_u$. For $r,s\in \{0,1\}$ let
\[
 S_{(i,j)}^{r,s}=S_{(i,j)}\setminus \{(2i+r,2j+s)\},
\]
as shown in Figure~\ref{fig:s_ij}.
Now fix  $\theta \in[0,1]$ and assume $u\in\Z^2$
with $u\equiv(r,s)\bmod 2$. Then define $\cL_u$ so that
\[
 \cL_u=
 \begin{cases}S_u^{r,s}&\text{with probability }1-\theta,\\
 S_u^{1-r,1-s}&\text{with probability }\theta.\\
 \end{cases}
\]
Figure~\ref{fig:z2-upper-bound} shows the situation when the first case always holds (i.e., when $\theta=0$). An equivalent definition is that we initially set $\cL_u=S_u^{r,s}$, and then independently with probability $\theta$ rotate each $\cL_u$ by $180^\circ$.

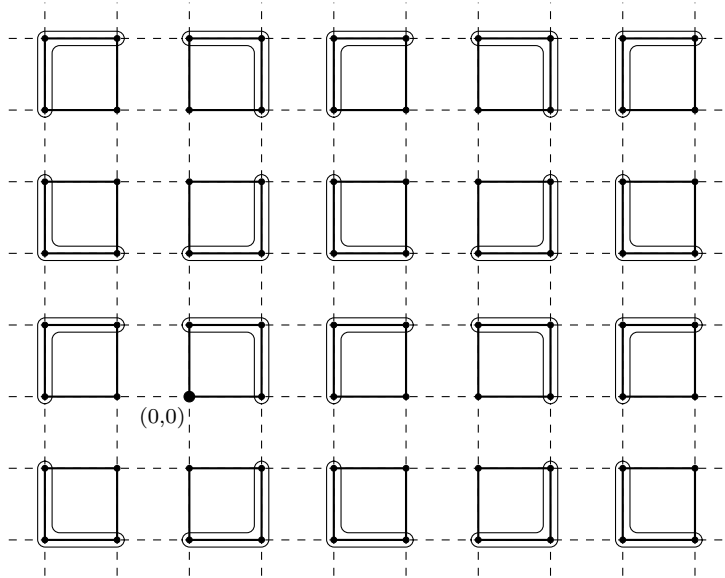
\begin{figure}[t]
\centering
\begin{tikzpicture}[scale=0.95]
 \clip (-0.5, -0.5) rectangle (9.5,7.5);
 \foreach \x in {0,...,7}
  \draw [dashed] (-0.5,\x)--(9+0.5,\x);
 \foreach \y in {0,...,9}
  \draw [dashed] (\y,-0.5)--(\y,0.5+7);
 \foreach \x in {0,...,9} \foreach \y in {0,...,7}
  \draw[fill] (\x,\y) circle (0.04);
 \foreach \x in {0,2,...,8} \foreach \y in {0,2,...,6}
  \draw[thick] (\x,\y) rectangle +(1,1);
 \foreach \x in {0,4,...,8} \foreach \y in {0,4,...,4} {
  \draw[rounded corners = 0.1cm] (\x+3.1,\y+2.1) -- (\x+3.1,\y+1.9) -- (\x+2.9,\y+1.9) -- (\x+2.9,\y+2.9) -- (\x+1.9,\y+2.9) -- (\x+1.9,\y+3.1) -- (\x+3.1,\y+3.1) -- (\x+3.1,\y+2.1);
  \draw[rounded corners = 0.1cm] (\x-0.1,2.1+\y) -- (\x-0.1,3.1+\y) -- (\x+1.1,3.1+\y) -- (\x+1.1,2.9+\y) -- (\x+0.1,2.9+\y) -- (\x+0.1,1.9+\y) -- (\x-0.1,1.9+\y) -- (\x-0.1,2.1+\y);
  \draw[rounded corners = 0.1cm] (\x-0.1,0.1+\y) -- (\x-0.1,1.1+\y) -- (\x+0.1,1.1+\y) -- (\x+0.1,0.1+\y) -- (\x+1.1,0.1+\y) -- (\x+1.1,-0.1+\y) -- (\x-0.1,-0.1+\y) -- (\x-0.1,0.1+\y);
  \draw[rounded corners = 0.1cm] (\x+2.1,-0.1+\y) -- (\x+1.9,-0.1+\y) -- (\x+1.9,0.1+\y) -- (\x+2.9,0.1+\y) -- (\x+2.9,1.1+\y) -- (\x+3.1,1.1+\y) -- (\x+3.1,-0.1+\y) -- (\x+2.1,-0.1+\y); }
 \node[style=vertex] at (2,2) {};
 \node[below left] at (2,2) {$\scriptstyle(0,0)\!\!$};
\end{tikzpicture}
\caption{A portion of the graph $\Z^2$ with the regions corresponding to the $\cL_u$ when $\theta=0$ highlighted. Dashed edges represent edges open with probability~$p'$, other edges are open with probability~$p$.}
\label{fig:z2-upper-bound}
\end{figure}

Now consider the $4\times 2$ rectangle $R = \{0,1,2,3\}\times\{0,1\}$ and let $E$ be the set of ten edges it induces. Let $\cC_0$ be the set of subsets of $E$ which give an open component containing at least two elements of each of $S_{(0,0)}^{0,1}$ and $S_{(1,0)}^{1,1}$, let $\cC_1$ be the analogous set for $S_{(0,0)}^{1,0}$ and $S_{(1,0)}^{1,1}$, $\cC_2$ the analogous set for $S_{(0,0)}^{0,1}$ and $S_{(1,0)}^{0,0}$, and $\cC_3$ the analogous set for $S_{(0,0)}^{1,0}$ and $S_{(1,0)}^{0,0}$. See Figure~\ref{fig:z2-rand-upper-bound} for an illustration of the target sets.

\begin{figure}
 \centering
 \tikzmath{\hgap = 1.5;}
 \tikzmath{\vgap = 2;}
 \tikzmath{\width = 6.2 + \vgap;}
 \begin{tikzpicture}[scale= 0.7*(\linewidth/1cm)/\width]
  \begin{scope}[shift={(0,0)}]
   \foreach \x in {0,...,3} {
    \draw node[style=vertex] at (\x,0) {};
    \draw node[style=vertex] at (\x,1) {};
    \draw (\x,0) -- (\x,1);
   }
   \draw (0,0) -- (1,0) (2,0) -- (3,0);
   \draw (0,1) -- (1,1) (2,1) -- (3,1);
   \draw[dashed] (1,0) -- (2,0) (1,1) -- (2,1);
   \draw[rounded corners = 0.1cm] (0.1 , -0.1) -- (-0.1, -0.1) -- (-0.1, 0.1) -- (0.9,0.1) -- ( 0.9, 1.1) -- ( 1.1 , 1.1) -- (1.1 , -0.1) -- (0.1, -0.1);
   \draw[rounded corners = 0.1cm] (1.9, 0.1) -- (1.9, 1.1) -- ( 2.1, 1.1) -- (2.1,0.1) -- ( 3.1, 0.1) -- ( 3.1, -0.1) -- (1.9, -0.1) -- (1.9, 0.1);
   \node[below=0.25cm] at (1.5, 0) {$\cC_0$};
  \end{scope}
  \begin{scope}[shift={(3+\hgap,0)}]
   \foreach \x in {0,...,3} {
    \draw node[style=vertex] at (\x,0) {};
    \draw node[style=vertex] at (\x,1) {};
    \draw (\x,0) -- (\x,1);
   }
   \draw (0,0) -- (1,0) (2,0) -- (3,0);
   \draw (0,1) -- (1,1) (2,1) -- (3,1);
   \draw[dashed] (1,0) -- (2,0) (1,1) -- (2,1);
   \draw[rounded corners = 0.1cm] (-0.1, 0.1) -- (-0.1, 1.1) -- (1.1, 1.1) -- (1.1, 0.9) -- (0.1,0.9) -- (0.1, -0.1) -- (-0.1, -0.1) -- (-0.1, 0.1);
   \draw[rounded corners = 0.1cm] (1.9, 0.1) -- (1.9, 1.1) -- ( 2.1, 1.1) -- (2.1,0.1) -- ( 3.1, 0.1) -- ( 3.1, -0.1) -- (1.9, -0.1) -- (1.9, 0.1);
   \node[below=0.25cm] at (1.5, 0) {$\cC_1$};
  \end{scope}
  \begin{scope}[shift={(0,-\vgap)}]
   \foreach \x in {0,...,3} {
    \draw node[style=vertex] at (\x,0) {};
    \draw node[style=vertex] at (\x,1) {};
    \draw (\x,0) -- (\x,1);
   }
   \draw (0,0) -- (1,0) (2,0) -- (3,0);
   \draw (0,1) -- (1,1) (2,1) -- (3,1);
   \draw[dashed] (1,0) -- (2,0) (1,1) -- (2,1);
   \draw[rounded corners = 0.1cm] (0.1 , -0.1) -- (-0.1, -0.1) -- (-0.1, 0.1) -- (0.9,0.1) -- ( 0.9, 1.1) -- ( 1.1 , 1.1) -- (1.1 , -0.1) -- (0.1, -0.1);
   \draw[rounded corners = 0.1cm] (2.1, 1.1) -- (3.1, 1.1) -- (3.1, -0.1) -- (2.9, -0.1) -- (2.9,0.9) -- (1.9, 0.9) -- (1.9, 1.1) -- (2.1, 1.1);
   \node[below=0.25cm] at (1.5, 0) {$\cC_2$};
  \end{scope}
  \begin{scope}[shift={(3+\hgap,-\vgap)}]
   \foreach \x in {0,...,3} {
    \draw node[style=vertex] at (\x,0) {};
    \draw node[style=vertex] at (\x,1) {};
    \draw (\x,0) -- (\x,1);
   }
   \draw (0,0) -- (1,0) (2,0) -- (3,0);
   \draw (0,1) -- (1,1) (2,1) -- (3,1);
   \draw[dashed] (1,0) -- (2,0) (1,1) -- (2,1);
   \draw[rounded corners = 0.1cm] (-0.1, 0.1) -- (-0.1, 1.1) -- (1.1, 1.1) -- (1.1, 0.9) -- (0.1,0.9) -- (0.1, -0.1) -- (-0.1, -0.1) -- (-0.1, 0.1);
   \draw[rounded corners = 0.1cm] (2.1, 1.1) -- (3.1, 1.1) -- (3.1, -0.1) -- (2.9, -0.1) -- (2.9,0.9) -- (1.9, 0.9) -- (1.9, 1.1) -- (2.1, 1.1);
   \node[below=0.25cm] at (1.5, 0) {$\cC_3$};
  \end{scope}
 \end{tikzpicture}
 \caption{The configurations corresponding to the definitions of $\cC_0$, $\cC_1$, $\cC_2$, and  $\cC_3$.}
\label{fig:z2-rand-upper-bound}
\end{figure}
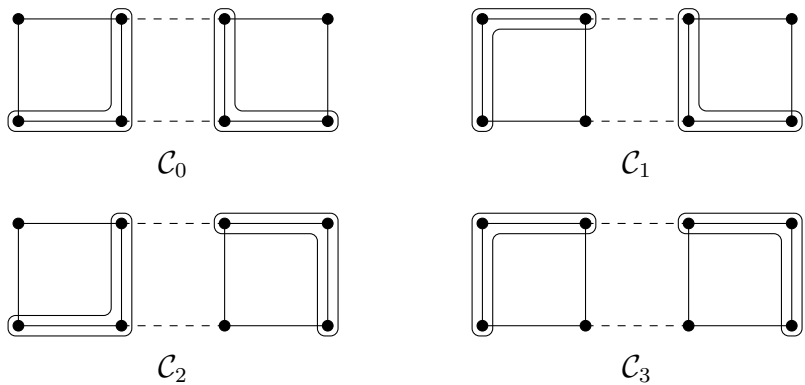

Given a bond $uv$ in the renormalised model, the relative states of $\cL_u$ and $\cL_v$ correspond
(up to symmetry) to one of the $\cC_i$. More specifically, if $uv$ lies in one of the $2\times2$ squares tiling $\Z^2$ in the renormalised model, then the pair
$(\cL_u,\cL_v)$ corresponds to the configurations $\cC_0$, $\cC_1$, $\cC_2$, $\cC_3$ with probabilities $(1-\theta)^2$, $\theta(1-\theta)$, $\theta(1-\theta)$, and $\theta^2$ respectively. Indeed, the default ($\theta=0$) case is just $\cC_0$ (up to symmetry) and then independently with probability $\theta$ we rotate each side by $180^\circ$. On the other hand, if
$uv$ joins two renormalised $2\times2$ squares then
in the default $\theta=0$ case we get a configuration that is $\cC_3$ up to symmetry. Thus in general we obtain configurations $\cC_0$, $\cC_1$, $\cC_2$, $\cC_3$ with probabilities $\theta^2$, $\theta(1-\theta)$, $\theta(1-\theta)$, and $(1-\theta)^2$ respectively.

Now, given probabilities $p_0$ and $p'_0$, consider a 1-independent model on $R$ in which edges inside $S_{(0,0)}$ and $S_{(1,0)}$ are open with probability at least $p_0$, and the two edges between $S_{(0,0)}$ and $S_{(1,0)}$ are open with probability at least $p'_0$.
Let the random set of open edges be $E' \subseteq E$. For fixed $\theta\in[0,1]$, let $p_1$ and $p'_1$ be such that, for all such models on~$R$,
\[
 p_1\leq (1-\theta)^2\Prb(E'\in\cC_0)+\theta(1-\theta)\Prb(E'\in\cC_1)+\theta(1-\theta)\Prb(E'\in\cC_2)+\theta^2\Prb(E'\in\cC_3)
\]
and
\[
 p'_1\leq \theta^2\Prb(E'\in\cC_0)+\theta(1-\theta)\Prb(E'\in\cC_1)+\theta(1-\theta)\Prb(E'\in\cC_2)+(1-\theta)^2\Prb(E'\in\cC_3).
\]
Then by symmetry, if our original 1-independent model on $\Z^2$ lies in $\cD_{\ge p_0,\ge p'_0}$ then the renormalised model lies in $\cD_{\ge p_1,\ge p'_1}$.

We will iterate this, giving a $\theta \in [0,1]$ and a sequence of pairs $(p_i,p'_i)$, $i=0,\dots,k$ such that, for every model in $\cD_{\ge p_i,\ge p'_i}$,
the renormalised model lies in $\cD_{\ge p_{i+1},\ge p'_{i+1}}$. The theorem will follow if we can exhibit such a sequence with $p_0=p'_0=0.8457$ and $\min\{p_k,p'_k\}\ge 0.8639$, as there is then almost surely an infinite open component in the $k$ times renormalised model. Note that at each stage we may assume that in the $i$ times renormalised model each edge is open with probability \emph{exactly} $p_i$ or $p'_i$ as appropriate. Indeed, independently deleting edges with the appropriate probabilities only makes the events $\cC_i$ less likely.  

After renormalising $i$ times, we wish to minimise the right-hand sides of the two inequalities above over all models in $\cD_{p_i,p'_i}$. As in the proof of Lemma~\ref{lem:q6-p}, we relax the condition that we have a 1-independent model on $R$ to obtain the following two linear programming problems.
\[\begin{array}{r@{}r@{\ }c@{\ }l}
 \text{Minimise:}&\multicolumn{3}{l}{
 (1-\theta)^2\sum_{S\in\cC_0}x_S+\theta(1-\theta)\sum_{S\in\cC_1}x_S+ \theta(1-\theta)\sum_{S\in\cC_2}x_S+\theta^2\sum_{S\in\cC_3}x_S}\\[3pt]
 \text{or:}&\multicolumn{3}{l}{
 \theta^2\sum_{S\in\cC_0}x_S+\theta(1-\theta)\sum_{S\in\cC_1}x_S+ \theta(1-\theta)\sum_{S\in\cC_2}x_S+(1-\theta)^2\sum_{S\in\cC_3}x_S}\\[6pt]
\text{Subject to:}
 &x_S   & \ge & 0,\\
 &y_S   & = &\sum_{T\supseteq S}x_T,\\
 &y_{S\cup\{e\}} & = &p_i\cdot y_S,\\
 &y_{S\cup\{f\}} & = &p'_i\cdot y_S,\\
 &y_{\emptyset}  & = &1,
\end{array}\]
where $S$ runs over all subsets of $E$, $e$ runs over all edges induced by $S_{(0,0)}$ or $S_{(1,0)}$ that are vertex-disjoint from all edges in $S$, and $f$ runs over the edges $\{(1,0),(2,0)\}$ and $\{(1,1),(2,1)\}$ that are vertex-disjoint from all edges in~$S$. The first of these optimisation problems yields a suitable value for $p_{i+1}$ and the second yields a suitable value for $p'_{i+1}$.

Running the above LP problems with $\theta=0.18$ using the \texttt{Gurobi} optimisation package we can find suitable values for $p_i$ and~$p'_i$, which we checked by confirming that the dual LP problem was feasible. The results are listed in Table~\ref{table:pi}. For $k=13$ we see that $p_k,p'_k\ge 0.8639$, and hence the model percolates by the results of~\cite{balister2005continuum}.
\end{proof}
    
\begin{table}
\centering
\begin{tabular}{cll} \toprule
\multicolumn{1}{c}{$i$} & \multicolumn{1}{c}{$p_i$} & \multicolumn{1}{c}{$p_i'$}\\
\midrule
0&0.845700&0.845700\\
1&0.859167&0.829055\\
2&0.856981&0.831846\\
3&0.857370&0.831456\\
4&0.857391&0.831616\\
5&0.857546&0.831779\\
6&0.857826&0.832114\\
7&0.858365&0.832753\\
8&0.859396&0.833976\\
9&0.861358&0.836303\\
10&0.865058&0.840691\\
11&0.871911&0.848815\\
12&0.884171&0.863343\\
13&0.904695&0.887637\\
14&0.934851&0.923277\\
\bottomrule
\end{tabular}
\caption{Suitable values for $p_i$ and $p'_i$ (with $\theta=0.18$).}
\label{table:pi}
\end{table}

We remark that there is no need to use the same value of $\theta$ for every renormalisation in the proof of Theorem~\ref{thm:z2-upper-bound}. Indeed, it seems rather unlikely that the optimal sequence of $\theta$ values is constant. We have, however, been unable to find a sequence that gives $\pmax(\Z^2) \leq 0.8456$.
Note that taking $\theta=0$ in the above, so that the $\cL_u$ are as illustrated in Figure~\ref{fig:z2-upper-bound}, yields only $\pmax(\Z^2)\leq 0.8463$. Moreover, the rather simpler homogeneous renormalisation where we take $\cL_u=S_u^{0,0}$ for all $u$, so that in the renormalised model we assume all edges are open with the same probability, gives $\pmax(\Z^2)\leq 0.8493$.

The renormalisation used in the proof of Theorem~\ref{thm:z2-upper-bound} can be combined with some ideas from~\cite{balister2005continuum} to give the following corollary.

\begin{corollary}\label{cor:z2strong}
 In any model in $\cD_{\ge 0.8457}(\Z^2)$ and for any $C>0$, the probability that there is an open path crossing the rectangle $[0,Cn]\times[0,n]$ from left to right tends to $1$ as $n\to\infty$. In particular, there are almost surely open cycles enclosing any bounded region of the plane.
\end{corollary}
\begin{proof}
We recall from Theorem~2 from~\cite{balister2005continuum} that a much simpler $2\times2$ renormalisation argument gives
that if the original model is in $\cD_{\ge 1-q}(\Z^2)$ then the renormalised model lies in $\cD_{\ge 1-10q^2}(\Z^2)$.
Indeed, suppose each edge is open with probability at least $1-q$ and take the good event $E_{uv}$ to be the event that there is a component meeting at least three out of the four vertices in each of $S_u$ and $S_v$. If this fails, we must have at least one of the following ten sets of edges closed\footnote{For each of the first three columns, having at least one of the two edges present in both upper and lower sets implies the corresponding $2\times2$ square has at least three vertices in the same component. The last two columns ensure that these components connect up.}
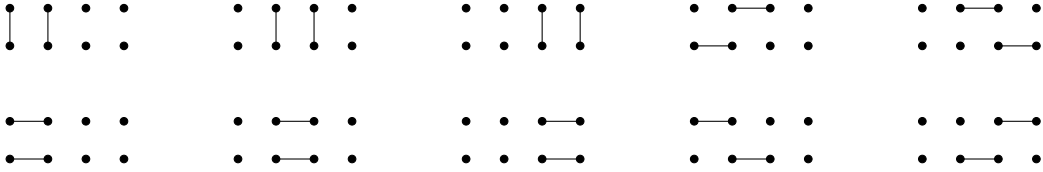
\[\begin{tikzpicture}[scale=0.5]
 \begin{scope}[shift={(0,0)}]
  \foreach \x in {0,...,3} \foreach \y in {0,...,1} \draw[fill] (\x,\y) circle (0.1);
  \draw (0,0) -- (0,1) (1,0)--(1,1);
 \end{scope}
 \begin{scope}[shift={(6,0)}]
  \foreach \x in {0,...,3} \foreach \y in {0,...,1} \draw[fill] (\x,\y) circle (0.1);
  \draw (1,0) -- (1,1) (2,0)--(2,1);
 \end{scope}
 \begin{scope}[shift={(12,0)}]
  \foreach \x in {0,...,3} \foreach \y in {0,...,1} \draw[fill] (\x,\y) circle (0.1);
  \draw (2,0) -- (2,1) (3,0)--(3,1);
 \end{scope}
 \begin{scope}[shift={(0,-3)}]
  \foreach \x in {0,...,3} \foreach \y in {0,...,1} \draw[fill] (\x,\y) circle (0.1);
  \draw (0,0) -- (1,0) (0,1)--(1,1);
 \end{scope}
 \begin{scope}[shift={(6,-3)}]
  \foreach \x in {0,...,3} \foreach \y in {0,...,1} \draw[fill] (\x,\y) circle (0.1);
  \draw (1,0) -- (2,0) (1,1)--(2,1);
 \end{scope}
 \begin{scope}[shift={(12,-3)}]
  \foreach \x in {0,...,3} \foreach \y in {0,...,1} \draw[fill] (\x,\y) circle (0.1);
  \draw (2,0) -- (3,0) (2,1)--(3,1);
 \end{scope}
 \begin{scope}[shift={(18,0)}]
  \foreach \x in {0,...,3} \foreach \y in {0,...,1}
  \draw[fill] (\x,\y) circle (0.1);
  \draw (0,0) -- (1,0) (1,1)--(2,1);
 \end{scope}
 \begin{scope}[shift={(24,0)}]
  \foreach \x in {0,...,3} \foreach \y in {0,...,1} \draw[fill] (\x,\y) circle (0.1);
  \draw (2,0) -- (3,0) (1,1)--(2,1);
 \end{scope}
 \begin{scope}[shift={(18,-3)}]
  \foreach \x in {0,...,3} \foreach \y in {0,...,1} \draw[fill] (\x,\y) circle (0.1);
  \draw (1,0) -- (2,0) (0,1)--(1,1);
 \end{scope}
 \begin{scope}[shift={(24,-3)}]
  \foreach \x in {0,...,3} \foreach \y in {0,...,1} \draw[fill] (\x,\y) circle (0.1);
  \draw (1,0) -- (2,0) (2,1)--(3,1);
 \end{scope}
\end{tikzpicture}\]
Hence, $\Prb(E_{uv})\ge 1-10q^2$. Note that this lower bound on edge probabilities also applies to the renormalisation used in the proof of Theorem~\ref{thm:z2-upper-bound}, since the good events there contain the good events here. Thus, once $p_k,p'_k>0.9$ in the proof of Theorem~\ref{thm:z2-upper-bound} (see Table~\ref{table:pi}), we
can inductively choose $p_{i+1}=p'_{i+1}=1-10q_i^2$
where $q_i=1-\min\{p_i,p'_i\}$ for all $i\ge k$. It is clear that then $p_i\to 1$ rapidly as $i\to\infty$.

Now given any $\eps>0$ we can find some $i_0$ such that $p_i,p'_i>1-\eps/\ceil{2C+1}$ for all $i\ge i_0$.
Suppose $n\ge 2^{i_0}$ and pick $i$ so that $2^i\le n<2^{i+1}$. Then in the $i$ times renormalised model we have an open path from $(-1,0)$ to $(\ceil{2C},0)$ 
with probability at least $1-\eps$. It follows that in the original model there is an open path in $\Z\times\{0,\dots,n\}$
starting before $x=0$ and ending at or after $x=\ceil{2C}2^i\ge Cn$ with probability at least $1-\eps$. Thus we have a path crossing $[0,Cn]\times[0,n]$ from left to right with probability at least $1-\eps$ as desired.

For the last part, fix a bounded set and let $m_0$ be large enough that $(-m_0,m_0)^2$ contains it. Then for all $m\geq m_0$, if there are open paths crossing the rectangles $[-2m,2m]\times[m,2m]$ and $[-2m,2m]\times[-2m,-m]$ from left to right and the rectangles $[-2m,-m]\times[-2m,2m]$ and $[m,2m]\times[-2m,2m]$ from top to bottom, then there is an open cycle enclosing the bounded set. By the above, in each case such a path is present with probability $1-o(1)$ as $m\to\infty$, so such an open cycle exists with probability~$1$. It follows by a union bound that, almost surely, for every bounded set there is an open cycle enclosing it.
\end{proof}

As noted in~\cite{balister2005continuum}, it follows from Corollary~\ref{cor:z2strong} that for every model in $\cD_{\geq0.8457}(\Z^2)$ the infinite component is almost surely unique: otherwise pick two vertices in different infinite components, then there is an open cycle enclosing both of them and this connects the two infinite components. It also follows that there is no infinite component in the dual graph associated to such models.

In Theorem~2 of~\cite{balister2005continuum} it was claimed that under every model in $\cD_{\geq 0.8639}$ the origin is in an infinite component with positive probability, but this result does not seem to follow from the proof given. However, we can now adapt the renormalisation used to prove Theorem~\ref{thm:z2-upper-bound} to show this claim holds for all models in $\cD_{\geq 0.8459}(\Z^2)$.

\begin{theorem}\label{thm:origin}
In any model in $\cD_{\ge 0.8459}(\Z^2)$ the origin is in an infinite open component with positive probability.
\end{theorem}
\begin{proof}
We shall inductively define nested events $F_k$ that imply that the origin is in a component of size at least $k+2$. Let $F_0$ be the event that the edge from $(0,0)$ to $(1,0)$ is open in the original model.
Now given $F_k$ in the $k$ times renormalised model
(which will imply the event that the edge from $(0,0)$ to $(1,0)$ in this model is open), we
first reflect the model in the line $x+y=1$ (i.e., map $(x,y)$ to $(1-y,1-x)$) and then renormalise as in the proof of Theorem~\ref{thm:z2-upper-bound}.
Note that this reflection preserves $\cD_{\ge p,\ge p'}$ as the $p$-edges and $p'$-edges are mapped onto edges of the same type. Also the event $F_k$ now implies that the
edge from $(1,0)$ to $(1,1)$ is open. Let $F_{k+1}$ be the event that $F_k$ holds and that the horizontal renormalised edge from $S_{(0,0)}$ to $S_{(1,0)}$ is open, i.e., that the event $E_{(0,0)(1,0)}$ (in the sense of the proof of Theorem~\ref{thm:z2-upper-bound}) holds. 

Note that any component meeting both $S_{(0,0)}$ and $S_{(1,0)}$ in large sets must also meet the edge from $(1,0)$ to $(1,1)$. Consider the component containing the origin in the restriction of the original model to the vertices corresponding to $\{(1,0),(1,1)\}$, and the component containing the origin in the restriction of the original model to the vertices corresponding to $S_{(0,0)}\cup S_{(1,0)}$. If $F_{k+1}$ holds, then the latter component strictly contains the former (as it contains some vertices from $S_{(1,0)}$) so by induction $F_{k+1}$ implies that the origin is in a component of size at least $k+3$ in the original model.

Now define $G_0=F_0$ and the event $G_{k+1}$ to be the event that, if the edge from $(1,0)$ to $(1,1)$ in the (reflected) $k$-fold renormalised model were open, then $E_{(0,0)(1,0)}$ would hold. In other words, $G_{k+1}$ is the event that adding the edge from $(1,0)$ to $(1,1)$
gives a suitable connected component in $S_{(0,0)}\cup S_{(1,0)}$. Note that $G_{k+1}$ depends only on the other nine edges
in $S_{(0,0)}\cup S_{(1,0)}$ and is clearly an increasing event as a function of these edges. Hence, if we lower bound $\Prb(G_{k+1})$ in $\cD_{p,p'}$
we also have the same bound in $\cD_{\ge p,\ge p'}$.
Now $F_{k+1}= F_k\cap G_{k+1}$, so by induction
\[
 \Prb(F_k)\ge 1-\sum_{i=0}^k\Prb(G_k^c).
\]
It is straightforward to modify the linear programming problem from the proof of Theorem~\ref{thm:z2-upper-bound} so as to minimise the probability of $G_k$ and hence find a lower bound on $F_k$ for small~$k$. For large $k$ we note that clearly $\Prb(G_k^c)\le q_k=1-p_k$, where the $k$-fold renormalised model is in $\cD_{\geq p_k,\geq p'_k}$,
and we have the inequality $q_{k+1}\le 10q_k^2$ as in the proof of Corollary~\ref{cor:z2strong}.

To show that the origin is in an infinite component with positive probability, it is sufficient to show that $\sum_{k=0}^\infty\Prb(G_k^c)<1$. Unfortunately, we cannot show this all the way down to $p=0.8457$, but we show it does hold for $p=0.8459$.
Using the results from Table~\ref{tab:origin} we have
\begin{align*}
 \sum_{k=0}^\infty\Prb(G_k^c)
 &\le \sum_{k=0}^{13}\Prb(G_k^c)+
 \sum_{k=14}^\infty q_k\\
 &\le 0.998359+ \sum_{i=1}^\infty 10^{2^i-1}q_{13}^{2^i}\\
 &\le 0.99836,
\end{align*}
where we have used that $q_{13}\le 0.0002$.
Hence, $\Prb(\cap_{k=0}^\infty F_k)>0$ and so the origin is in an infinite component with positive probability.
\end{proof}

\begin{table}
\centering
\begin{tabular}{clll} \toprule
\multicolumn{1}{c}{$i$} & \multicolumn{1}{c}{$p_i$} & \multicolumn{1}{c}{$p_i'$}&\multicolumn{1}{c}{$\Prb(G_k^c)\le$}\\
\midrule
0&0.845900&0.845900&0.154100\\
1&0.859515&0.829480&0.096201\\
2&0.857661&0.832648&0.097540\\
3&0.858670&0.832999&0.096787\\
4&0.859879&0.834568&0.095945\\
5&0.862289&0.837404&0.094255\\
6&0.866795&0.842751&0.091100\\
7&0.875072&0.852561&0.085314\\
8&0.889637&0.869816&0.075168\\
9&0.913248&0.897752&0.058828\\
10&0.945814&0.936217&0.036503\\
11&0.978577&0.974824&0.014314\\
12&0.996611&0.996024&0.002247\\
13&0.999914&0.999899&0.000057\\
\bottomrule
\end{tabular}
\caption{Bounds used in Theorem~\ref{thm:origin} (with $\theta=0.18$).}
\label{tab:origin}
\end{table}

\section{Open problems}\label{sec:open-probs}

In this section we discuss some interesting related problems on 1-independent percolation models on hypercubes and lattices. We begin by restating a question first posed by Day, Falgas-Ravry, and Hancock in~\cite{day2020long}.

\begin{problem}[\cite{day2020long}]\label{prob:qn-conn}
For $n\geq 3$, what is the largest $p = p(n)$ for which there is a model in $\cD_{\geq p}(Q_n)$ under which the hypercube is always disconnected?
\end{problem}

The simple construction used in the proof of Theorem~1.4 in~\cite{balister2012critical} shows that this~$p$ satisfies $p \geq 1/2$, but is this best possible? The case $n=2$ of Problem~\ref{prob:qn-conn}, where $Q_n=C_4$, was answered in~\cite{day2020long} where they showed that the maximum is indeed $p=1/2$ (this follows immediately from considering the expected number of edges in a model with edge probability greater than $1/2$).

It would also be interesting to determine the models above the threshold which minimise the probability that $Q_n$ is connected. When $n = 2$, it is not difficult to show that these models are exactly those supported on subgraphs with either two or four edges, with
\[
 \newcommand\verts{\draw[fill](0,0)circle(.1)(1,0)circle(.1)(0,1)circle(.1)(1,1)circle(.1);}
 \Prb(\tikz[scale=0.25]{\verts\draw(0,0)--(0,1)--(1,1)--(1,0)--(0,0);})=2p-1,\quad
 \Prb(\tikz[scale=0.25]{\verts\draw(0,0)--(0,1) (1,1)--(1,0);})=
 \Prb(\tikz[scale=0.25]{\verts\draw(0,0)--(1,0) (1,1)--(0,1);})=q^2,\quad 
 \Prb(\tikz[scale=0.25]{\verts\draw(0,0)--(1,0)--(1,1);})=
 \Prb(\tikz[scale=0.25]{\verts\draw(0,0)--(0,1)--(1,1);})=\alpha,\quad
 \Prb(\tikz[scale=0.25]{\verts\draw(0,1)--(0,0)--(1,0);})=
 \Prb(\tikz[scale=0.25]{\verts\draw(0,1)--(1,1)--(1,0);})=pq-\alpha,
\]
where $q=1-p$ and $\alpha\in[0,pq]$.

The answer is already unknown for $n = 3$, although computer experiments suggest an answer for large~$p$. Let the \emph{signs model} on a graph $G$ be the 1-independent percolation model where each vertex of $G$ is independently assigned the sign `$+$' with probability $\theta$ and the sign `$-$' otherwise. An edge is open if both its endpoints have the same sign (so each edge is open with probability $p=\theta^2 + (1 - \theta)^2$). Computational evidence strongly suggests that this model minimises the probability that $Q_3$ is connected when $p > 0.55$. 
However, for smaller $p$ there are models where connectivity is less likely, and indeed for $p<0.516$ one can find models that
are disconnected almost surely. 

\begin{conjecture}
For $p>0.55$, among models in $\cD_{\geq p}(Q_3)$ the signs model minimises the probability that $Q_3$ is connected.
\end{conjecture}

We also hypothesise that for general $n\geq 3$, the signs model minimises the probability that $Q_n$ is connected when $p$ is large enough. It should be noted that sadly none of the models listed above which minimise the probability of connectivity when $n=2$ are signs models (unless $p=1$).

\begin{problem}
Is it true that for $n\ge3$, if $p$ is sufficiently close to~$1$, then among models in $\cD_{\geq p}(Q_n)$ the signs model minimises the probability that $Q_n$ is connected? If so, how close to~$1$ does~$p$ need to be?
\end{problem}

If the signs model is indeed optimal on $Q_3$, say for $p\geq 0.5720$, then this leads to improved bounds on $\pgiant$ and $\lim_{n\to\infty}\pmax(\Z^n)$ as follows. Start by applying the first step of the renormalisation in the proof of Theorem~\ref{thm:cube}, with $k=3$, to a model in $\cD_{\geq p}(Q_n)$ to obtain a model on $Q_{n-3}$ in which the probability that an edge is open given both its endpoints are open is at least
\begin{equation*}
    1-\frac{2^{16}(1-p)^8}{\Big(\big(1+\sqrt{2p-1}\big)^8+\big(1-\sqrt{2p-1}\big)^8\Big)^2}.
\end{equation*}
This expression is greater than $p$ when $p\geq 0.5720$ so for such $p$ we can apply this renormalisation repeatedly until we obtain a constant $I=I(p)$ and a model in $\cD(Q_{n-I},p_v,p_e)$ where $p_v>0$ and $p_e>18/19$ such that the existence of a giant component in this model implies the existence of a giant component in the original. It follows from Proposition~\ref{prop:cube} that there exists $C=C(p)>0$ such that with high probability this model, and hence the original model, has a component containing at least a $C$ proportion of its vertices. Thus, if the signs model is optimal on $Q_3$ for $p\geq 0.5720$, then $\pgiant\leq 0.5720$. Similarly, $\lim_{n\to\infty}\pmax(\Z^n)\leq 0.5720$ under the same assumption.

We have shown that when each edge of the hypercube is open with probability at least $0.5847$ (and $n$ is large), there is a component containing a constant fraction of the vertices with high probability, so in particular $\pgiant\leq 0.5847$. The best lower bound on this threshold is $1/2$ (using the model from the proof of Theorem~1.4 in~\cite{balister2012critical} again). It was conjectured in~\cite{falgas-ravry20231-independent} that this is tight, that is, $\pgiant=1/2$: they conjectured that for all $p>1/2$, under every model in $\cD_{\geq p}(Q_n)$ the hypercube contains a component of size at least $\big(\frac{1+\sqrt{2p-1}}{2}-o(1)\big)\cdot2^n$ with probability $1-o(1)$. 

\begin{problem}
    What is the value of $\pgiant$?
\end{problem}

It is not clear to us whether the threshold probabilities for always disconnecting $Q_n$ (as in Problem~\ref{prob:qn-conn}) have a limit as $n\to\infty$, but if they do it would be interesting to know whether this limit differs from $\pgiant$.

We have improved the bounds on $\pmax(\Z^2)$ and $\lim_{n\to\infty}\pmax(\Z^n)$ so that they now stand at
\[\begin{array}{l@{\ }c@{\ }c@{\ }c@{\ }l}
    0.555197\dots &\leq &\pmax(\Z^2)&\leq &0.8457,\\
    0.535898\dots &\leq &\lim_{n\to\infty}\pmax(\Z^n) &\leq& 0.5847.
\end{array}\]
There remains a large gap between the upper and lower bounds in both cases and it would be interesting to reduce either of these gaps. We note that Result~\ref{res:high-conf} shows that the lower bound for $\pmax(\Z^2)$ is unlikely to be correct, and suggests that the correct value of $\pmax(\Z^2)$ should be higher than $\lim_{n\to\infty}\pmax(\Z^n)$. We close by restating the fundamental Questions~\ref{question:Z2} and~\ref{question:Zn} from above, recalling that the authors of~\cite{falgas-ravry20231-independent} conjecture that $\pmax(\Z^n)=4-2\sqrt{3}$ for some $n\geq 3$.

\questionZtwo*
\questionZn*

\paragraph{Acknowledgements}
We would like to thank the anonymous referee for their helpful comments.

\bibliographystyle{scott.bst}
\bibliography{refs}

\end{document}